\documentclass[12pt]{article}
\usepackage{srcltx} 
\usepackage{amsmath,amssymb,amsthm,mathtools,enumitem}
\usepackage{color}
\usepackage{authblk}
\usepackage{geometry}
\newtheorem{Theorem}{Theorem}[section]

\newtheorem{Lemma}{Lemma}[section]

\newtheorem{Definition}{Definition}[section]

\begin{document}

\def\eR{\mathbf{R}}
\def\Rd{{\eR}^d}
\def\Rdd{{\eR}^{d\times d}}
\def\Rdsym{{\eR}^{d\times d}_{sym}}
\def\eN{\mathbf{N}}
\def\eZ{\mathbf{Z}}

\def\dd{\mbox{d}}
\newcommand{\essinf}{\operatorname{ess\,inf}}
\newcommand{\esssup}{\operatorname{ess\,sup}}
\newcommand{\supp}{\operatorname{supp}}
\def\d{\; \mathrm{d}}
\def\dx{\; \mathrm{d}x}
\def\dy{\; \mathrm{d}y}
\def\dz{\; \mathrm{d}z}
\def\diff{\mathsf{d}}
\def\div{\mathop{\mathrm{div}}\nolimits}
\def\diam{\mathrm{diam}}

\def\Aeps{\mathbf{A}^\varepsilon}
\def\ueps{\mathbf{u}^\varepsilon}
\def\boldp{\mathbf{p}}
\def\bolds{\mathbf{s}}
\def\boldu{\mathbf{u}}
\def\boldv{\mathbf{v}}
\def\boldw{\mathbf{w}}
\def\boldz{\mathbf{z}}
\def\boldA{\mathbf{A}}
\def\boldB{\mathbf{B}}
\def\boldD{\mathbf{D}}
\def\boldE{\mathbf{E}}
\def\boldF{\mathbf{F}}
\def\boldI{\mathbf{I}}
\def\boldO{\mathbf{O}}
\def\boldP{\mathbf{P}}
\def\boldT{\mathbf{T}}
\def\boldU{\mathbf{U}}
\def\boldUeps{\mathbf{U}^\varepsilon}
\def\boldV{\mathbf{V}}
\def\boldW{\mathbf{W}}
\def\boldZ{\mathbf{Z}}
\def\bzero{\mathbf{0}}
\def\uk{\boldu^{k}}
\def\boldAk{\boldA^k}
\def\balpha{\boldsymbol{\alpha}}
\def\bbeta{\boldsymbol{\beta}}
\def\boldzeta{\boldsymbol{\zeta}}
\def\boldeta{\boldsymbol{\eta}}
\def\boldchi{\boldsymbol{\chi}}
\def\boldxi{\boldsymbol{\xi}}
\def\bphi{\boldsymbol{\varphi}}
\def\bpsi{\boldsymbol{\psi}}
\def\tbpsi{\tilde{\bpsi}}
\def\WCon{\xrightharpoonup{\hphantom{2-s}}}
\def\WSCon{\xrightharpoonup{\raisebox{0pt}[0pt][0pt]{\hphantom{\(\scriptstyle{2-s}\)}}}^*}
\def\WTSSCon{\xrightharpoonup{\raisebox{0pt}[0pt][0pt]{\(\scriptstyle{2-s}\)}}^*}
\def\STSCon{\xrightarrow{2-s}}
\def\SCon{\xrightarrow{\hphantom{2-s}}}
\def\ModConv#1{\xrightarrow{\mathmakebox[1.5em]{#1}}}
\def\ModConvM{\ModConv{M}}
\def\NabOv{\overline\nabla}

\title{Existence and homogenization of nonlinear elliptic systems in nonreflexive  spaces\thanks{M.~Bul\'{\i}\v{c}ek was partially supported
by the Czech Science Foundation (grant no. 16-03230S). The research of A.~\'Swierczewska--Gwiazda and P.~Gwiazda have received funding from the National Science Centre, Poland, 2014/13/B/ST1/03094. This work was partially supported by the Simons - Foundation grant 346300 and the Polish Government MNiSW 2015-2019 matching fund.}}
\author[$\dagger$]{Miroslav Bul\'\i\v cek}
\author[$\star$]{Piotr Gwiazda}
\author[$\dagger$]{Martin Kalousek}
\author[$\star$]{Agnieszka \'Swierczewska-Gwiazda}

\affil[$\dagger$]{Mathematical Institute, Faculty of Mathematics and Physics, Charles University\\ Sokolovsk\'{a} 83,
186 75 Praha 8, Czech Republic}
\affil[$\star$]{Institute of Applied Mathematics, University of Warsaw, ul. Banacha 2, 02-097 Warsaw, Poland}
\date{}
\maketitle
\abstract{We consider a strongly nonlinear elliptic problem with the homogeneous Dirichlet boundary condition. The growth and the coercivity of the elliptic operator is assumed to be indicated by an inhomogeneous anisotropic $\mathcal{N}$-function. First, an existence result is shown under the assumption that the $\mathcal{N}$--function or its convex conjugate satisfies $\Delta_2$--condition. The second result concerns the homogenization process for families of strongly nonlinear elliptic problems with the homogeneous Dirichlet boundary condition under above stated conditions on the elliptic operator, which is additionally assumed to be periodic in the spatial variable.
}

\section{Introduction}
Given $\boldF\!:\Omega\rightarrow\eR^{d\times N}$ and a nonlinear operator $\boldA\!: \Rd\times \eR^{d\times N}\rightarrow \eR^{d\times N}$ we study elliptic systems  of the form
\begin{equation}\label{StudPr}
\begin{alignedat}{2}
\div \boldA\left(\frac{x}{\varepsilon},\nabla\ueps\right)& =\div\boldF &&\text{ in }\Omega,\\
\ueps&=0 &&\text{ on }\partial\Omega,
\end{alignedat}
\end{equation}
where $\Omega\subset\Rd$ is a bounded Lipschitz domain with $d\geq 2$ and $\ueps\!:\Omega \to \eR^N$ with $N\in \mathbf{N}$ is an unknown.
Our goal is twofold: Firstly, we want to show the solvability, i.e., the existence of $\ueps$  of \eqref{StudPr} for as general class of operators  $\boldA$ as possible and secondly,  for operators $\boldA$ that are $Y$-periodic with respect to the first variable, where $Y:=(0,1)^d$, we want to study the limit process as $\varepsilon \to 0$.

The basic framework we are dealing in, or more precisely, the class of operators $\boldA$ we are interested in, is the following:
\begin{enumerate}[label=(A\arabic*)]
	\item\label{AO} $\boldA$ is a Carath\'eodory mapping, i.e., $\boldA(\cdot,\xi)$ is measurable for any $\xi\in \eR^{d\times N}$ and $\boldA(y,\cdot)$ is continuous for a.a. $y\in\Rd$,
	\item\label{Aperi} $\boldA$ is $Y-$periodic, i.e., periodic in each argument $y_i,i=1,\ldots,d$ with the period $1$,
	\item\label{ATh} There exists an ${\mathcal N}-$function $M\!:\Rd\times \eR^{d\times N}\rightarrow[0,\infty)$ and a constant $c>0$ such that for a.a. $y\in Y$ and all $\boldxi\in \eR^{d\times N}$ there holds\footnote{  Note that the condition could be formulated more generally, i.e., $\boldA(y,\boldxi)\cdot\boldxi\geq c(M(y,\boldxi)+M^*(y,\boldA(y,\boldxi)))-k(y)$ for some integrable function $k$.  For readability we omit this generality here, however such case could easily be treated, see e.g.~\cite{GSG08}. }
\begin{equation*}
		\boldA(y,\boldxi)\cdot\boldxi\geq c(M(y,\boldxi)+M^*(y,\boldA(y,\boldxi))),
\end{equation*}
	\item\label{AF} For all $\boldxi,\boldeta\in\eR^{d\times N}$ such that $\boldxi\neq\boldeta$ and a.a. $y\in Y$, we have
\begin{equation*}
		(\boldA(y,\boldxi)-\boldA(y,\boldeta))\cdot(\boldxi-\boldeta)> 0.
\end{equation*}
\end{enumerate}
The conditions \ref{AO}, \ref{ATh} and \ref{AF} describe a general monotone spatially dependent operator, and will be assumed in the existence result of the paper. The periodicity assumption \ref{Aperi} will be used for the homogenization process. We also refer the reader to Appendix~\ref{Ape1} for the notion of $\mathcal{N}$--functions.

The first problem we want to solve is the existence and  uniqueness of the solution to \eqref{StudPr}, which is an elliptic problem for which many results are available. In particular, as far as the function $M$ and the function $M^*$ satisfy $\Delta_2$--condition, the existence and  uniqueness of a solution directly follows from the Minty method\footnote{When   $M$ and the function $M^*$ satisfy $\Delta_2$--condition then the corresponding function spaces are reflexive and separable and therefore the classical methods work.}. Therefore, our main interest is to investigate the case when $\Delta_2$--condition is not valid. However, in such setting one has to overcome the difficulties caused by the non-reflexivity and non-separability of related function spaces. The  result related to such a setting was established in \cite{GSZG2017}, where the existence  of a solution was shown for operators $\boldA$ satisfying our assumption with one proviso: the  function~$M$ has to be log-H\"{o}lder continuous with respect to the spatial variable. An analogous  result for parabolic problem was shown in \cite{SG2014, SG2014a}. Hence, in this paper, we want to deal also with possibly discontinuous functions~$M$, which can be well motivated by many physically relevant applications. To mention a few, where function $M$ is discontinuous with respect to spatial variable and has even the exponential growth with respect to the gradient of unknown, we refer to \cite{S1,S2,S3}.  For this case, the first existence result was obtained in \cite{GSG208}, where the authors assumed in addition that the conjugate function $M^*$ satisfies the $\Delta_2$--condition. Surprisingly, the method developed in \cite{GSG208} cannot be simply adapted to the case when $M$ satisfies $\Delta_2$--condition. Nevertheless, in this paper we will overcome this difficulty and by using the \emph{dual approach} we shall obtain the following, kind of unifying, result for $\mathcal{N}$--function possibly discontinuous with respect to the spatial variable and not fulfilling the $\Delta_2$--condition.
\begin{Theorem}\label{Thm:MainOne}
Let $N\geq 1$, $\Omega\subset\Rd$ be a bounded Lipschitz domain with $d\ge 2$. Assume that an operator $\boldA$ satisfies \ref{AO}, \ref{ATh} and \ref{AF}. Let $M:\Omega\times\eR^{d\times N}\rightarrow[0,\infty)$ be the ${\mathcal N}$--function from \eqref{ATh} satisfying  $\Delta_2$--condition and for all $R>0$ there holds
\begin{equation}\label{Cis1}
	\int_\Omega\sup_{|\boldxi|=R}M(x,\boldxi)\dx<\infty,
\end{equation}
or $M^*$ satisfies  $\Delta_2$--condition and for all $R>0$
\begin{equation}\label{Cis2}
	\int_\Omega\sup_{|\boldxi|=R}M^*(x,\boldxi)\dx<\infty.
\end{equation}
Then there exists a unique weak solution to the problem
\begin{equation}\label{EllPr}
	\begin{alignedat}{2}
		\div\boldA(x,\nabla\boldu(x))&=\div\boldF(x)&&\text{ in }\Omega,\\
		\boldu&=0&&\text{ on }\partial\Omega.
	\end{alignedat}
\end{equation}
\end{Theorem}
The theorem above is stated vaguely on purpose without precise definition of function spaces and related problems. For the rigorous statement we refer to Lemma~\ref{Lem:ExUnMStDTwo} and Lemma~\ref{Lem:ExMDTwo} in Section~\ref{SS3}.

The second main goal of the paper is the homogenization process as $\varepsilon \to 0_+$. For this purpose we employ in addition the periodicity assumption~\ref{Aperi} (but without any requirement on continuity) and we shall also require certain uniform control on the corresponding $\varepsilon$--dependent $\mathcal{N}$--functions $M(\frac{x}{\varepsilon}, \boldxi)$, which is of the form:
\begin{enumerate}[label=(M\arabic*)]
	\item\label{MO} $M$ is $Y-$periodic in the first variable,
		\item\label{MTh} there exist ${\mathcal N}-$functions $m_1,m_2:[0,\infty)\rightarrow[0,\infty)$ such that
	\begin{equation*}
			m_1(|\boldxi|)\leq M(y,\boldxi)\leq m_2(|\boldxi|)\text{ on }Y.
	\end{equation*}
\end{enumerate}


The studies on homogenization of elliptic equations go back to the works of Oleinik and Zhikov \cite{OZ82} and Allaire \cite{Al92}. The setting of non-standard growth conditions  of the operator $\boldA$ already  appeared in  \cite{ZP11}, where the authors considered the growth  prescribed by means of variable exponent $p(x)$, so the corresponding function spaces were varying with respect to $\varepsilon\to 0$ in the homogenization process. Notice that in $L^{p(x)}$ setting they required that $1<p_{\min}\le p(x)\le p_{\max} < \infty$, so the corresponding functions spaces were reflexive and separable as well. The first attempt to deal with the $\mathcal{N}$--function not satisfying $\Delta_2$--condition, was done in \cite{BGKSG17}, where for the operator $\boldA$ fulfilling \ref{AO}--\ref{AF} and the function $M$ satisfying \ref{MO}--\ref{MTh} the limit $\varepsilon\to 0$ was successfully established provided that $M$ is log--H\"{o}lder continuous with respect to the first variable.

In this paper, we shall overcome this difficulty and show that even for discontinuous functions $M$ one can obtain the fairy complete theory provided that $M$ or $M^*$ satisfy $\Delta_2$ condition, but without any assumption on the continuity with respect to the spatial variable. Indeed, inspired by  \cite{ZP11,BGKSG17}, we show that the limit $\boldu$ of a sequence of solutions to \eqref{StudPr} satisfies a problem, in which the nonlinear operator is independent of a spatial variable, i.e., the problem possesses the form
\begin{equation}\label{HPr}
	\begin{alignedat}{2}	
			\div\hat\boldA(\nabla \boldu)&=\div\boldF&&\text{ in }\Omega,\\
			\boldu&=0 &&\text{ on }\partial\Omega,
	\end{alignedat}
\end{equation}
where, denoting $Y:=(0,1)^d$, the operator $\hat\boldA$ is defined as
\begin{equation*}
	\hat\boldA(\boldxi):=\int_Y\boldA(y,\boldxi+\boldW(y))\dy,
\end{equation*}
and $\boldW$ is the solution of the cell problem, i.e., $\boldW:=\nabla\boldw$ with $Y$-periodic $\boldw:\Rd \to \eR^N$ solving
\begin{equation*}
	\div\boldA(y,\boldxi+\nabla \boldw(y))=0\text{ in }Y.
\end{equation*}
Notice, that the existence and the uniqueness of the solution  to the cell problem can be obtained by a straightforward modification of the proof of Theorem~\ref{Thm:MainOne}. The second main result of the paper then reads as follows.
\begin{Theorem}\label{Thm:MainTwo}
Let $\boldA$ satisfy \ref{AO}-\ref{AF}, the ${\mathcal N}-$function $M$ satisfy \ref{MO}-\ref{MTh} and let at least one of the following hold:
\begin{enumerate}[label=(C\arabic*)]
	\item \label{MDeltaTwo} the $\mathcal{N}$--function $M$ satisfies $\Delta_2$--condition,
 \item \label{MStDeltaTwo} $M^*$, the convex conjugate $\mathcal{N}$--function  to $M$, satisfies $\Delta_2$--condition.
\end{enumerate}
Furthermore, assume that
\begin{equation}\label{Assumption:F}
\boldF\in L^{\infty}(\Omega; \eR^{d\times N})
\end{equation}
and for any $\varepsilon>0$ let $\ueps$  be a unique solution to the problem \eqref{StudPr}. Then for an arbitrary sequence $\{\varepsilon_j\}_{j=1}^\infty$ such that $\varepsilon_j\rightarrow 0$ as $j\rightarrow\infty$ we have the following convergence result
\begin{equation*}
	\boldu^{\varepsilon_j}\rightharpoonup \boldu\text{ in }W^{1,1}_0(\Omega;\eR^N),
\end{equation*}
where $\boldu^{\varepsilon_j}$ is the sequence of solutions solving  \eqref{StudPr} with $\varepsilon=\varepsilon_j$ and $\boldu$ is a unique solution to \eqref{HPr}, provided that either the considered problem is scalar, i.e., $N=1$, or the set $\Omega$ is star-shaped or the embedding
\begin{equation}\label{SobMOEmb}
	W^1_0 L^{m_1}(\Omega)\hookrightarrow L^{m_2}(\Omega) \textrm{ and }\ W^1 L^{m_1}(Y)\hookrightarrow L^{m_2}(Y)
\end{equation}
hold true.
\end{Theorem}

The paper is organized as follows. In Subsection~\ref{FSp} we introduce the function spaces corresponding to our setting, Subsection~\ref{homo} is related to introduction of tools used for homogenization and in Subsection~\ref{mapp} the properties of the homogenized operator $\hat\boldA$ are discussed. Section~\ref{SS3} is devoted to the proof of Theorem~\ref{Thm:MainOne} in the case that $M$ satisfies $\Delta_2$--condition and \eqref{Cis1}, while Section~\ref{Sec:1.2} is devoted  to the proof of Theorem~\ref{Thm:MainTwo}. Finally, for the sake of reader's convenience, we collect all other tools and known results needed in the paper in Appendices. Appendix~\ref{Ape1} is devoted to the introduction of the general Musielak--Orlicz spaces, Appendix~\ref{Ape2} to certain functional--analytic tools and Appendix~\ref{Ape3} to the part of the proof of Theorem~\ref{Thm:MainOne} in case that $M^*$ satisfies $\Delta_2-$condition and \eqref{Cis2} is true.

\section{Preliminaries}
This section is devoted to the preliminary observations and to the introduction of tools needed later for proofs of the main results of the paper. In Subsection~\ref{FSp} we introduce the function spaces related to the problem we are interested in. Next, in Subsection~\ref{homo}, we recall tools used in the homogenization theory and finally, in Subsection~\ref{mapp}, we establish the properties of the homogenized operator $\hat\boldA$ as well as the properties of the related function space.

\subsection{Function spaces related to the problem}\label{FSp}
Since we deal with rather general function spaces and  growth conditions imposed on the nonlinearity $\boldA$, we recall in Appendix~\ref{Ape1} several facts about the Musielak--Orlicz spaces $L^M$ and $E^M$,  we refer the interested reader to \cite{S05,SI69} for more details. Throughout this section, $\Omega\subset\eR^d$ will be a bounded domain and $Y:=(0,1)^d$. For an $\mathcal{N}$--function $M:Y\times \eR^{d\times N}\to \eR_+$, we use the subscript $y$ to underline  the role of $y$ for the Musielak--Orlicz spaces $L^{M_y}(\Omega\times Y;\eR^{d\times N})$ and $E^{M_y}(\Omega\times Y;\eR^{d\times N})$, which we  endow with the Luxembourg norm
\begin{equation*}
	\|\boldv\|_{L^{M_y}}=\|\boldv\|_{E^{M_y}}:=\inf\left\{\lambda>0:\int_{\Omega}\int_Y M\left(y,\frac{\boldv(x,y)}{\lambda}\right)\dy \dx\leq 1\right\}.
\end{equation*}
We note that whenever a function dependent on a variable from $Y$ appears, it is always $Y-$periodic although the $Y-$periodicity might not be emphasized. We further denote the spaces of smooth periodic or compactly supported functions and their solenoidal analogues as
$$
\begin{aligned}
C^\infty_{per}(Y;\eR^{N})&:=\{\boldv\in C^\infty(\Rd; \eR^{N}): \boldv\ \text{is}\ Y\text{-periodic}\},\\
C^\infty_{c}(\Omega;\eR^{N})&:=\{\boldv\in C^\infty(\Rd; \eR^{N}): \textrm{supp }\boldv \textrm{ is compact in } \Omega\},\\
C^\infty_{per,\div}(Y;\eR^{N})&:=\{\boldv\in C^\infty_{per}(Y; \eR^{N}): \div\boldv=0\text{ in }Y\},\\
C^\infty_{c,\div}(\Omega;\eR^{N})&:=\{\boldv\in C_c^\infty(\Omega; \eR^{N}): \div\boldv=0\text{ in }\Omega\}
\end{aligned}
$$
and naturally also the corresponding Bochner spaces $C^\infty_c\left(\Omega;C^\infty_{per}(Y)\right)$. Then the standard Sobolev spaces are defined as
$$
W^{1,1}_0(\Omega;\eR^N):=\overline{\{\boldv\in C^{\infty}_c(\Omega;\eR^N)\}}^{\|\cdot \|_{1,1}}, \qquad W^{1,1}_{per}(Y;\eR^N):=\overline{\{\boldv\in C^\infty_{per}(Y;\eR^N); \, \int_{Y}\boldv =0\}}^{\|\cdot \|_{1,1}}.
$$
Moreover, due to the Poincar\'{e} inequality, we always choose an equivalent norm on $W^{1,1}_0$ and $W^{1,1}_{per}$ as $\|\boldv\|_{1,1}:=\|\nabla \boldv\|_1$. Further, we  define the corresponding Sobolev--Musielak--Orlicz spaces and we also distinguish among them similarly as in the case of Musielak--Orlicz spaces. That means, we introduce the following notation
\begin{align*}
W^1_0E^M(\Omega;\eR^N)&:=\overline{C^\infty_c(\Omega;\eR^N)}^{\|\cdot \|_{W^{1}_0L^M(\Omega)}}, \quad W^1_{per}E^M(Y):=\overline{\left\{\boldv\in C^\infty_{per}(Y;\eR^N): \int_Y\boldv=0\right\}}^{\|\cdot\|_{W^{1}_{per}L^M(Y)}},
\end{align*}
where we endow the spaces with the norm\footnote{The fact, that it is indeed a norm, is a direct consequence of the Poincar\'{e} inequality in $W^{1,1}$.} $\|\boldv\|_{W^1_0 L^M(\Omega)}:=\|\nabla\boldv\|_{L^M(\Omega)}$ and $\|\boldv\|_{W^1_{per}L^M(Y)}:=\|\nabla\boldv\|_{L^M(Y)}$. In addition, we introduce the weak$^*$ closures of the above defined spaces, i.e.,
\begin{align*}
	W^{1}_0L^M(\Omega;\eR^N):=\{&\boldu\in W^{1,1}_0(\Omega;\eR^N) : \nabla\boldu\in L^M(\Omega;\eR^{d\times N});\; \exists \{\boldu^n\}_{n=1}^{\infty}\subset C^\infty_c(\Omega;\eR^N): \\&\nabla\boldu^n\WSCon\nabla\boldu\text{ in }L^M(\Omega;\eR^{d\times N})\},\\
	W^{1}_{per}L^M(Y;\eR^N):=\{&\boldu\in W^{1,1}_{per}(Y;\eR^N):  \nabla\boldu\in L^M(Y;\eR^{d\times N});\; \exists \{\boldu^n\}_{n=1}^{\infty}\subset C^\infty_{per}(Y;\eR^N),\int_Y\boldu^n=0: \\&\nabla\boldu^n\WSCon\nabla\boldu\text{ in }L^M(Y;\eR^{d\times N})\}.
\end{align*}
The above spaces are usually referred as the Sobolev--Musielak--Orlicz spaces. However, as will be shown later, these spaces can be  too small in principle and therefore we introduce a different class of Sobolev--Musielak--Orlicz spaces by
\begin{align*}
	V^M_0&:=\left\{\boldv\in W^{1,1}_0(\Omega;\eR^N): \nabla \boldv\in L^M(\Omega;\eR^{d\times N})\right\},\\
	V^M_{per}&:=\left\{\boldv\in W^{1,1}_{per}(Y;\eR^N): \nabla \boldv\in L^M(Y;\eR^{d\times N})\right\}.
\end{align*}
These spaces will be again equipped with the norms $\|\boldv\|_{V_0^M}=\|\nabla\boldv\|_{L^M(\Omega)}$ and $\|\boldv\|_{V_{per}^M}=\|\nabla\boldv\|_{L^M(Y)}$, which makes them Banach spaces. Finally, we define the spaces of  mappings having zero divergence as
\begin{align*}
	E^M_{\div}(\Omega;\eR^{d\times N}):=\overline{\{C^\infty_{\div}(\Omega;\eR^{d\times N})\}}^{\|\cdot\|_{L^M(\Omega)}}, \quad E^M_{per,\div}(Y;\eR^{d\times N}):=\overline{\{C^\infty_{per,\div}(Y;\eR^{d\times N})\}}^{\|\cdot\|_{L^M(Y)}},
\end{align*}
and
\begin{align*}
	L^M_{\div}(\Omega;\eR^{d\times N}):=\{&\boldT\in L^M(\Omega;\eR^{d\times N});\; \exists \{\boldT^n\}_{n=1}^{\infty}\subset E^M_{\div}(\Omega;\eR^{d\times N}): \\&\boldT^n\WSCon\boldT\text{ in }L^M(\Omega;\eR^{d\times N})\},\\
	L^M_{per,\div}(Y;\eR^{d\times N}):=\{&\boldT\in L^{M}_{per}(Y;\eR^{d\times N});\; \exists \{\boldT^n\}_{n=1}^{\infty}\subset E^M_{per,\div}(Y;\eR^{d\times N}): \\&\boldT^n\WSCon\boldT\text{ in }L^M(Y;\eR^{d\times N})\},
\end{align*}
which are again Banach spaces.
Further, we utilize the following closed subspaces of $E^M(Y;\eR^{d\times N})$, $E^{M^*}(Y;\eR^{d\times N})$ respectively, and their annihilators
\begin{equation*}
	\begin{split}
		G&:=\{\nabla \boldw: \boldw\in W_{per}^1E^M(Y;\eR^{N}) \},\\
		G^\bot&:=\{\boldW\in L_{per}^{M^*}(Y;\eR^{d\times N}):\int_Y\boldW(y)\cdot\boldV(y)\dy=0\text{ for all }\boldV\in G\},\\
		D&:=E^{M^*}_{per,\div}(Y;\eR^{d\times N}),\\
		D^\bot&:=\{\boldW\in L_{per}^{M}(Y;\eR^{d\times N}):\int_Y\boldW(y)\cdot\boldV(y)\dy=0\text{ for all }\boldV\in D\}.
	\end{split}
\end{equation*}
We note that
\begin{equation}\label{DBotChar}
	D^\bot=\{\nabla\boldv:\boldv\in V^M_{per}\},
\end{equation}
which can be proven by modifying the procedure from Step 5 of the proof of Lemma~\ref{Lem:ExUnMStDTwo} from Appendix~\ref{Ape3}.

Finally, we introduce the notation for the second annihilators. However, at this point we have to already distinguish the cases whether $M$ or $M^*$ satisfy $\Delta_2$--condition. Hence, if $M^*$ satisfies $\Delta_2$--condition, then we know that $(L^{M^*})^*=L^M$ and we can observe that
\begin{equation}\label{FEqExpression}
G^{\bot\bot}=\{\boldW\in L^M(Y;\eR^{d\times N}): \int_Y \boldW(y)\cdot\boldV(y)\dy=0\ \text{ for all }\boldV\in G^{\bot}(Y)\}.
\end{equation}
Similarly, if $M$ satisfies $\Delta_2$--condition, then $(L^M)^*= L^{M^*}$ and we obtain
\begin{equation}\label{Dkk}
	D^{\bot\bot}=\{\boldW\in L^{M^*}_{per}(Y;\eR^{d\times N}): \int_Y \boldW(y)\cdot\boldV(y)\dy=0\ \text{ for all }\boldV\in D^\bot\}.
\end{equation}

\subsection{Standard tools used  for homogenization} \label{homo}
This section is devoted to the introduction of the two--scale convergence by means of periodic unfolding. This approach allows one to represent the weak two--scale convergence in terms of the standard weak convergence in a Lebesgue space on the product $\Omega\times Y$, details for the case of $L^p$ spaces can be found in~\cite{V06}. In the same manner the strong two--scale convergence is introduced. Since function spaces, which we are working with, provide only the weak$^*$ compactness of bounded sets, we introduce the weak$^*$ two--scale compactness. However, it turns out that this notion of convergence and some of its properties are sufficient for our purposes. We define functions $n:\eR\rightarrow \eZ$ and $N:\Rd\rightarrow \eZ^d$ as
	\begin{equation*}
			n(t)=\max\{n\in \eZ: n\leq t\}\ \forall t\in\eR,	N(x)=(n(x_1),\ldots, n(x_d)),\ \forall x\in\Rd.
	\end{equation*}
	
Then we have for any $x\in \Rd,\varepsilon>0$, a two--scale decomposition $x=\varepsilon\left(N\left(\frac{x}{\varepsilon}\right)+R\left(\frac{x}{\varepsilon}\right)\right)$, where $R$ is the reminder function. We also define for any $\varepsilon>0$ a two--scale composition function $S_\varepsilon:\Rd\times Y\rightarrow\Rd$ as $S_\varepsilon(x,y):=\varepsilon \left(N\left(\frac{x}{\varepsilon}\right)+y\right)$.
It follows immediately that
	\begin{equation}\label{TSCompUnifConv}
		S_\varepsilon(x,y)\rightarrow x\text{ uniformly in }\Rd\times Y\text{ as }\varepsilon\rightarrow 0
	\end{equation}
	since $S_\varepsilon(x,y)=x+\varepsilon\left(y-R\left(\frac{x}{\varepsilon}\right)\right)$. In the rest of the section we assume that $m$ is an ${\mathcal N}-$function.
\begin{Definition}
We say that a sequence of functions $\{v^\varepsilon\}\subset L^m(\Rd)$
\begin{enumerate}
	\item converges to $v^0$ weakly$^*$ two--scale in $L^m(\Rd\times Y)$, $v^\varepsilon\WTSSCon v^0$, if $v^\varepsilon\circ S_\varepsilon$ converges to $v^0$ weakly$^*$ in $L^m(\Rd\times Y)$,
	\item converges to $v^0$ strongly two--scale in $E^m(\Rd\times Y)$, $v^\varepsilon\STSCon v^0$, if $v^\varepsilon\circ S_\varepsilon$ converges to $v^0$ strongly in $E^m(\Rd\times Y)$.
\end{enumerate}
\end{Definition}
We define two--scale convergence in $L^m(\Omega\times Y)$ as two--scale convergence in $L^m(\Rd\times Y)$ for functions extended by zero to $\Rd\setminus\Omega$. The following lemma will be utilized to express properties of two--scale convergence in terms of single-scale convergence.
\begin{Lemma}[Lemma 1.1, \cite{V06}]\label{Lem:Decomp}
	Let $g$ be measurable with respect to a $\sigma-$algebra generated by the product of the $\sigma-$algebra of all Lebesgue--measurable subsets of $\eR^d$ and the $\sigma-$algebra of all Borel--measurable subsets of $Y$. Assume in addition that $g\in L^1(\eR^d;L^\infty_{per}(Y))$ and extend it by $Y-$periodicity to $\eR^d$ for a.a. $x\in\eR^d$. Then, for any $\varepsilon>0$, the function $(x,y)\mapsto g(S_\varepsilon(x,y),y)$ is integrable and
	\begin{equation*}
		\int_{\Rd}g\left(x,\frac{x}{\varepsilon}\right)\dx=\int_{\Rd}\int_Y g(S_\varepsilon(x,y),y)\dy\dx.
	\end{equation*}
\end{Lemma}
Several useful properties of the two--scale convergence are summarized in the following lemma.
\begin{Lemma}[Lemma~2.6, \cite{BGKSG17}]\label{Lem:Facts2S}
Assume that $m:[0,\infty)\rightarrow[0,\infty)$ is an ${\mathcal N}-$function.
\begin{enumerate}[label=(\roman*)]
	\item\label{F2SFir} Let $v:\Omega\times Y\rightarrow\eR$ be  Carath\'eodory, $v\in E^m(\Omega\times Y)$, $v$ be $Y-$periodic, define $v^\varepsilon(x)=v(x,\frac x\varepsilon)$ for $x\in\Omega$. Then $v^\varepsilon\STSCon v$ in $E^m(\Omega\times Y)$ as $\varepsilon\rightarrow 0$.
	\item\label{F2SS} Let $v^\varepsilon\WTSSCon v^0$ in $L^m(\Omega\times Y)$ then $v^\varepsilon\rightharpoonup^* \int_Y v^0(\cdot,y)\dy$ in $L^m(\Omega)$.
		\item\label{F2ST} Let $v^{\varepsilon}\WTSSCon v^0$ in $L^{m}(\Omega\times Y)$ and $w^{\varepsilon}\STSCon w^0$ in $E^{m^*}(\Omega\times Y)$ then $\int_\Omega v^{\varepsilon} w^{\varepsilon}\to\int_\Omega\int_Yv^0w^0$.
	\item\label{F2SFo} Let $v^{\varepsilon}\WTSSCon v^0$ in $L^m(\Omega\times Y)$ then for any $\psi \in C^\infty_c\left(\Omega;C^\infty_{per}(Y)\right)$
	\begin{equation*}
		\lim_{\varepsilon\rightarrow 0}\int_\Omega v^{\varepsilon}(x) \psi\left(x,\frac{x}{\varepsilon}\right)\dx=\int_\Omega\int_Y v^0(x,y)\psi(x,y)\dy\dx.
	\end{equation*}
	\item\label{F2SFi} Let $\{v^\varepsilon\}$ be a bounded sequence in $L^m(\Omega)$. Then there is $v^0\in L^m(\Omega\times Y)$ and a sequence $\varepsilon_k\to0$ as $k\to\infty$ such that $v^{\varepsilon_k}\WTSSCon v^0$ in $L^m(\Omega\times Y)$ as $k\to\infty$.
	\item\label{F2SSi} Let $\{v^{\varepsilon}\}$ converge weakly$^*$ to $v$ in $W_0^1L^m(\Omega)$. Then $v^{\varepsilon}\WTSSCon v$ in $L^m(\Omega\times Y)$ and there is a sequence $\varepsilon_k\rightarrow 0$ as $k\rightarrow \infty$ and $\mathbf{v}\in L^m(\Omega\times Y;\Rd)$ such that $\nabla v^{\varepsilon_k}\WTSSCon\nabla v+\boldv$ in $L^m(\Omega\times Y;\Rd)$ as $k\to\infty$ and
	\begin{equation*}
		\int_Y \mathbf{v}(x,y)\cdot \bpsi(y)\dy=0
	\end{equation*}
 for a.a. $x\in\Omega$ and any $\bpsi\in C^\infty_{per}(Y;\Rd),\div \bpsi=0$ in $Y$.	
\item\label{F2SSev}  Let $\Phi:\Rd\times \eR^{d\times N} \rightarrow \eR$ satisfy:
\begin{enumerate}[label=(\alph*)]
	\item $\Phi$ is Carath\'eodory,
	\item $\Phi(\cdot,\boldxi)$ is $Y-$periodic for any $\boldxi\in\eR^{d\times N}$, $\Phi(y,\cdot)$ is convex for almost all $y\in Y$,
	\item $\Phi\geq 0$, $\Phi(\cdot,0)=0$.
\end{enumerate}
If $\boldUeps\WTSSCon \boldU$ in $L^m(\Omega\times Y;\eR^{d\times N})$ then
\begin{equation*}
\liminf_{\varepsilon\rightarrow \infty}\int_{\Omega}\Phi\left(\frac{x}{\varepsilon},\boldUeps(x)\right)\dx\geq \int_{\Omega\times Y}\Phi(y,\boldU(x,y))\dy\dx.
\end{equation*}
\end{enumerate}
\end{Lemma}
\subsection{Properties of the cell problem} \label{mapp}
In this subsection, we investigate the properties of the homogenized operator $\hat\boldA$. Let us define an operator $\hat{\boldA}\!:\eR^{d\times N}\rightarrow\eR^{d\times N}$ as
\begin{equation}\label{AhDef}
	\hat{\boldA}(\boldxi)=\int_Y\boldA(y,\boldxi+ \nabla \boldw_{\boldxi})\dy
\end{equation}
where the $Y-$periodic function $\boldw_{\boldxi}$ is a unique solution of the following cell problem
\begin{equation}\label{CellPr}
	\div\boldA(y,\boldxi+\nabla \boldw_{\boldxi})=0\text{ in }Y.
\end{equation}
The most important properties of the operator are summarized in the following lemma.
\begin{Lemma}\label{L2.3}
	Let $Y=(0,1)^d$, the operator $\boldA$ satisfy \ref{AO}-\ref{ATh}, the ${\mathcal N}-$function $M$ satisfy \ref{MO}-\ref{MTh}. Assume that at least one of conditions \ref{MDeltaTwo} and \ref{MStDeltaTwo} holds. 	Then for arbitrary $\boldxi\in \eR^{d\times N}$, the problem \eqref{CellPr} admits a unique weak solution $\boldw_{\boldxi}$ such that
	\begin{itemize}
		\item[] if \ref{MStDeltaTwo} holds then $\boldw_{\boldxi}\in W^1_{per}L^M(Y;\eR^N)$ and
			\begin{equation}\label{CellPrWFMStDTw}
				\int_{Y} \boldA\left(y,\boldxi+\nabla \boldw_{\boldxi}(y)\right)\cdot\nabla\bphi(y)\dy=0 \qquad \textrm{ for all } \bphi\in W^1_{per}L^M\left(Y;\eR^{N}\right);
			\end{equation}
			\item[] if \ref{MDeltaTwo} holds then $\boldw_{\boldxi}\in V_{per}^M$ and
			\begin{equation}\label{CellPrWFMDTw}
				\int_{Y} \boldA\left(y,\boldxi+\nabla \boldw_{\boldxi}(y)\right)\cdot\nabla\bphi(y)\dy=0 \qquad \textrm{ for all }\bphi\in V_{per}^M.
			\end{equation}
	\end{itemize}
Moreover,
\begin{equation}\label{WSCont}
	\boldxi^k\rightarrow \boldxi\text{ in }\eR^{d\times N}\text{ implies }\boldA(\cdot,\boldxi^k+\nabla \boldw_{\boldxi^k})\WSCon\boldA(\cdot,\boldxi+\nabla \boldw_{\boldxi})\text{ in }L^{M^*}(Y;\eR^{d\times N}),
\end{equation}
where $\boldw_{\boldxi^k}$ is a solution of the cell problem corresponding to $\boldxi^k$ and $\boldw_{\boldxi}$ to $\boldxi$, respectively.
\begin{proof}
	 The existence and uniqueness of solution $\boldw_{\boldxi}$ can be obtained by a straightforward modification of Lemma~\ref{Lem:ExMDTwo} and Lemma~\ref{Lem:ExUnMStDTwo}, respectively.  We note that in the case when $M$ satisfies $\Delta_2$ condition we first deal with the existence of solution $\boldT\in L^{M^*}_{per,\div}(Y;\eR^{d\times N})$ of the problem
\begin{equation*}
			\int_Y (\boldB(y,\boldT(y))-\boldxi)\cdot\boldW(y)\dy=0 \textrm{ for all }\boldW\in E^{M^*}_{per,\div}(Y;\eR^{d\times N}),
\end{equation*}
where $\boldB$ is the inverse operator to $\boldA$ and there exists a sequence $\{\boldT^k\}_{k=1}^{\infty}\subset E^{M^*}_{per,\div}(Y;\eR^{d\times N})$ such that
$$
\boldT^k\WSCon \boldT \textrm{ in }  L^{M^*}(Y; \eR^{d\times N}).
$$
Having introduced such a $\boldT$, we  realize that there is $\boldw_{\boldxi}\in V^M_{per}$ such that $\nabla\boldw_{\boldxi}=\boldB(\cdot,\boldT)-\boldxi$ which solves~\eqref{CellPrWFMDTw}.
		
Now, we focus on \eqref{WSCont}. Let us assume that $\{\boldxi^k\}_{k=1}^\infty$ is such that $\boldxi^k\rightarrow\boldxi$ in $\eR^{d\times N}$ as $k\rightarrow\infty$. For simplicity, we abbreviate $\boldw^k$ the solution of the cell problem corresponding to $\boldxi^k$ and as $\boldw$ the solution corresponding to $\boldxi$. We also denote $\boldZ^k(y):=\boldA(y,\boldxi^k+\nabla \boldw^k(y))$. First, we show that
		\begin{equation}\label{ZNEst}
			\int_YM(y,\boldxi^k+\nabla \boldw^k)+M^*(y,\boldZ^k(y))\dy\leq c.
		\end{equation}
		Since $\boldw^k$ is an admissible test function in \eqref{CellPrWFMStDTw}, \eqref{CellPrWFMDTw} respectively, with $\boldxi=\boldxi^k$, we obtain
		\begin{equation}\label{TestCPrSol}
			\int_Y \boldZ^k\cdot\nabla \boldw^k=0.
		\end{equation}
		Using \ref{ATh}, \eqref{TestCPrSol} and the Young inequality yields		
		\begin{equation*}
			\begin{split}
			c&\int_Y M^*(y,\boldZ^k(y))+M(y,\boldxi^k+\nabla \boldw^k)\dy\leq \int_Y \boldZ^k\cdot(\boldxi^k+\nabla \boldw^k)\dy=\int_Y \boldZ^k\cdot\boldxi^k\dy\\
			&\leq \frac{c}{2}\int_Y M^*(y,\boldZ^k)\dy+\int_Y M\left(y,\frac{2}{c}\boldxi^k\right)\dy,
			\end{split}
		\end{equation*}
		the second integral on the right hand side is finite due to \ref{MTh} as $\{\boldxi^k\}_{k=1}^\infty$ is bounded.	Without loss of generality, we can assume that
		\begin{equation}\label{WSConWNZN}
				\boldZ^k\WSCon\boldZ\text{ in }L^{M^*}_{per}(Y;\eR^{d\times N})
		\end{equation}
as $k\rightarrow\infty$. We shall show that $\boldZ=\boldA(\cdot,\boldxi+\nabla \boldw)$ and that \eqref{CellPrWFMStDTw} or \eqref{CellPrWFMDTw}, respectively, holds true.
		
Indeed, let $\boldw$ be a weak solution corresponding to $\boldxi$, which exists according to the first part of the Lemma. Using the monotonicity of $\boldA$, the fact that $\boldw$ and $\boldw^k$ are solutions and uniform estimate~\eqref{ZNEst}, we have
$$
\begin{aligned}
&\int_Y \left| (\boldA (y, \boldxi^k+\nabla \boldw^k(y)) - \boldA (y,\boldxi + \nabla \boldw(y)))\cdot (\boldxi^k+\nabla \boldw^k(y)-\boldxi - \nabla \boldw(y) )\right| \dy\\
&\quad =\int_Y (\boldA (y, \boldxi^k+\nabla \boldw^k(y)) - \boldA (y,\boldxi + \nabla \boldw(y)))\cdot (\boldxi^k+\nabla \boldw^k(y)-\boldxi - \nabla \boldw(y) ) \dy\\
&\quad =\int_Y (\boldA (y, \boldxi^k+\nabla \boldw^k(y)) - \boldA (y,\boldxi + \nabla \boldw(y)))\cdot (\boldxi^k - \boldxi ) \dy \le C|\boldxi^k - \boldxi|.
\end{aligned}
$$
Thus, letting $k\to \infty$ and using the strict monotonicity of $\boldA$, see \ref{AF}, we obtain
$$
\nabla \boldw^k \to \nabla \boldw \qquad \textrm{a.e. in } Y
$$
and due to the continuity of $\boldA$ with respect to the second variable, we get that $\boldZ(y)=\boldA(y,\boldxi+\nabla \boldw(y))$. Since this solution is unique, we obtain  that not only a subsequence extracted from $\{\boldZ^k\}_{k=1}^\infty$  converges weakly$^*$ to $\boldA(\cdot,\boldxi+\nabla \boldw)$ in $L^{M^*}(Y;\eR^{d\times N})$ but also the whole sequence converges to the same limit, which finishes the proof of \eqref{WSCont}.
\end{proof}
\end{Lemma}

We finish this subsection by introducing the function spaces related to the homogenized operator $\hat\boldA$. These function spaces and their properties rely on the fact whether $M$ or $M^*$ satisfy $\Delta_2$--condition. Therefore, all results below will be always split into two cases. In case that $M^*$ satisfies $\Delta_2$--condition, we just follow \cite{BGKSG17} and state all results without proofs. On the other hand, since the case when $M$ satisfies $\Delta_2$--condition is different, we provide all details for this situation.

We start with the function spaces related to $\hat\boldA$ in case that $M^*$ satisfies $\Delta_2$--condition.
We define a functional $f:\eR^{d\times N}\rightarrow [0,\infty)$ as
\begin{equation}\label{FDefinition}
	f(\boldxi)=\inf_{\boldW\in G}\int_YM(y,\boldxi+\boldW(y))\dy.
\end{equation}
Then, the basic properties of $f$ are stated in next lemma.
\begin{Lemma}[Lemma 2.8, \cite{BGKSG17}]\label{Lem:FProp}
	Let the $\mathcal{N}$--function satisfy \ref{MO}-\ref{MTh}. Then the functional $f$ is an ${\mathcal N}-$function, i.e., it satisfies:
	\begin{enumerate}[label=\arabic*)]
		\item \label{fFi} $f(\boldxi)=0$ if and only if $\boldxi=\bzero$,
		\item \label{fS} $f(\boldxi)=f(-\boldxi)$,
		\item  $f$ is convex,
		\item \label{fFo} $\lim_{|\boldxi|\rightarrow 0}\frac{f(\boldxi)}{|\boldxi|}=0$, $\lim_{|\boldxi|\rightarrow\infty}\frac{f(\boldxi)}{|\boldxi|}=\infty$.
	\end{enumerate}
	\end{Lemma}
%

Finally, we state the key property of $f$ provided that $M^*$ satisfies $\Delta_2$--condition. Note that this lemma will play the essential role in the homogenization process.
\begin{Lemma}\label{Lem:FEqExpr}
	Let the $\mathcal{N}$--function satisfy \ref{MO}-\ref{MTh}, $M^*$ satisfy $\Delta_2$--condition and $f$ be defined by \eqref{FDefinition}. Then $f$ can be alternatively expressed as
	\begin{equation}\label{FAlter}
		f(\boldxi)=\inf_{\boldW\in G^{\bot\bot}(Y)}\int_Y M(y,\boldxi+\boldW(y))\dy.
	\end{equation}
	\begin{proof}
First, according to \cite[Lemma 2.9]{BGKSG17}, we have the following expression for the conjugate function $f^*$
\begin{equation}\label{FSt}
		f^*(\boldxi)=\inf_{\substack{\boldW^*\in G^\bot,\\ \int_Y \boldW^*(y)\dy=\boldxi}}\int_Y M^*(y,\boldW^*(y))\dy.
	\end{equation}
Next, we compute $f^{**}:=(f^*)^*$, which is the second conjugate to $f$. Defining a functional $\mathcal{G}$ as
		\begin{equation*}
			\mathcal{G}(\boldW)=\int_Y M^*(y,\boldW(y))\dy
		\end{equation*}
		one justifies that $\mathcal{G}$ is closed, continuous at $\bzero\in G^\bot$ and the fact that
		\begin{equation*}
			\mathcal{G}^*(\boldW^*)=\int_Y M(y,\boldW^*(y))\dy
		\end{equation*}
		analogously to the justification of all these facts for the functional $\mathcal{F}$ in the proof of \cite[Lemma 2.9]{BGKSG17}, or see also the analog in the proof of Lemma~\ref{Lem:HStEqExpr}.
		Then we compute
		\begin{equation*}
			\begin{split}
			f^{**}(\boldxi)&=\sup_{\boldeta\in \eR^{d\times N}}\left\{\boldxi\cdot\boldeta-\inf_{\boldW\in G_0^\bot}\mathcal{G}(\boldeta+\boldW)\right\}\\
			&=\sup_{\boldeta\in\eR^{d\times N}}\left\{-\inf_{\boldW\in G^\bot_0}\left\{ \mathcal{G}(\boldeta+\boldW)-\int_{Y}\boldxi\cdot(\boldeta+\boldW(y))\dy\right\}\right\}\\
			&=-\inf_{\boldeta\in\eR^{d\times N}}\left\{\inf_{\boldW\in G^\bot_0}\left\{\mathcal{G}(\boldeta+\boldW)-\int_{Y}\boldxi\cdot(\boldeta+\boldW(y))\dy\right\}\right\}\\
			&=-\inf_{\boldV\in \eR^{d\times N}\oplus G^\bot_0}\left\{\mathcal{G}(\boldV)-\int_{Y}\boldxi\cdot\boldV(y)\dy\right\}\\&=-\inf_{\boldV\in G^\bot}\left\{\mathcal{G}(\boldV)-\int_{Y}\boldxi\cdot\boldV(y)\dy\right\}=\inf_{\boldU\in G^{\bot\bot}(Y)}\mathcal{G}^*(\boldxi+\boldU)\\
			&=\inf_{\boldU\in G^{\bot\bot}(Y)}\int_Y M(y,\boldxi+\boldU(y))\dy,
			\end{split}
		\end{equation*}
		where the last equality follows by Lemma~\ref{Lem:Duality}. Then we immediately conclude \eqref{FAlter} because $f=f^{**}$ as $f$ is convex and lower semicontinuous.
		
	\end{proof}
\end{Lemma}

Next, in order to be able to prove the main theorem in the case when $M$ satisfies $\Delta_2$--condition we introduce a functional $h^*:\eR^{d\times N}\rightarrow[0,\infty)$ as
\begin{equation}\label{hSt:Def}
	h^*(\boldxi)=\inf_{\boldW\in D_0}\int_Y M^*(y,\boldxi+\boldW(y))\dy,
\end{equation}
where
\begin{equation}\label{D0}
D_0:=\{\boldW \in D: \; \int_Y \boldW(y)\dy=0\}.
\end{equation}
The properties of $h^*$ are summarized in the ensuing lemma. Since the proof of each property is analogous to the proof of corresponding property of $f$ in Lemma~\ref{Lem:FProp}, the proof of the lemma is omitted.
\begin{Lemma}\label{Lem:HProp}
	Let the $\mathcal{N}$--function $M$ satisfy \ref{MO}-\ref{MTh} and $h^*$ be defined by \eqref{hSt:Def}. Then the functional $h^*$ is an ${\mathcal N}-$function, i.e., it satisfies:
	\begin{enumerate}[label=\arabic*)]
		\item \label{hFi} $h^*(\boldxi)=0$ if and only if $\boldxi=\bzero$,
		\item \label{hS} $h^*(\boldxi)=h^*(-\boldxi)$,
		\item  $h^*$ is convex,
		\item \label{hFo} $\lim_{|\boldxi|\rightarrow 0}\frac{h^*(\boldxi)}{|\boldxi|}=0$, $\lim_{|\boldxi|\rightarrow\infty}\frac{h^*(\boldxi)}{|\boldxi|}=\infty$.
	\end{enumerate}
\begin{proof}
		First, we show that
		\begin{equation}\label{AbBeEstF}
			m_2^*(|\boldxi|)\leq h^*(\boldxi)\leq m_1^*(|\boldxi|).
		\end{equation}
To do that, we first observe that it follows from \ref{MTh} and the fact that all involved functions are $\mathcal{N}$--functions that for all $y\in Y$ and all $\boldxi\in \eR^{d\times N}$ we have
\begin{equation}\label{AbBeEstFHLP}
			m_2^*(|\boldxi|)\leq M^*(y,\boldxi)\leq m_1^*(|\boldxi|).
		\end{equation}
Let us show the first inequality in \eqref{AbBeEstF}. Using \eqref{AbBeEstFHLP}, Jensen's inequality (applied to the convex function $m_2^*$) and the definition of $D_0$, we have
		\begin{equation*}
		\begin{split}
			h^*(\boldxi)&=\inf_{\boldW\in D_0}\int_Y M^*(y,\boldxi+\boldW(y))\dy\geq \inf_{\boldW\in D_0}\int_Y m_2^*(|\boldxi+\boldW(y)|)\dy\geq \inf_{\boldW\in D_0}m_2^*\left(\left|\boldxi+\int_Y\boldW(y)\dy\right|\right)\\
			&= m_2^*(|\boldxi|).
			\end{split}
		\end{equation*}
On the other hand, since $\bzero \in D_0$ we also have
\begin{equation*}
\begin{split}
			h^*(\boldxi)&=\inf_{\boldW\in D_0}\int_Y M^*(y,\boldxi+\boldW(y))\dy\le \int_Y M^*(y,\boldxi)\dy \le \int_Y m_1^*(|\boldxi|)\dy= m_1^*(|\boldxi|).
\end{split}
\end{equation*}
The assertions \ref{fFi} and \ref{fFo} then follow immediately from \eqref{AbBeEstF} and the facts that $m_1$ and $m_2$ are $\mathcal{N}$--functions.
The property \ref{fS} directly follows from the fact that  $M$ is even in the second argument and  $D_0$ is a subspace of $E^{M^*}_{per}(Y;\eR^{d\times N})$.
In order to show the convexity of $h^*$ we take an arbitrary $\lambda\in (0,1)$, $\boldxi_1,\boldxi_2\in\eR^{d\times N}$ and $\boldW_1,\boldW_2\in D_0$. Again, since $D_0$ is a subspace of $E^{M^*}_{per}(Y;\eR^{d\times N})$, we can use the definition of $h^*$  and the convexity of $M^*$ to obtain
\begin{equation*}
\begin{split}
h^*(\lambda\boldxi_1+(1-\lambda)\boldxi_2)&\leq \int_Y M^*(y,\lambda(\boldxi_1+\boldW_1(y))+(1-\lambda)(\boldxi_2+\boldW_2(y)))\dy\\
&\leq \lambda\int_Y  M^*(y,\boldxi_1+ \boldW_1(y))\dy+(1-\lambda)\int_Y M^*(y,\boldxi_2+\boldW_2(y))\dy.
\end{split}
\end{equation*}
		One obtains the desired conclusion by taking the infimum over $\boldW_1$ and $\boldW_2$ on the right hand side of the latter inequality.
\end{proof}	
\end{Lemma}

Next, we show the counter part of Lemma~\ref{Lem:FEqExpr} stated now for $h^*$ and the case when $M$ satisfies $\Delta_2$--condition.
\begin{Lemma}\label{Lem:HStEqExpr}
	Let the $\mathcal{N}$--function $M$ satisfy \ref{MO}-\ref{MTh} and $\Delta_2$--condition and $h^*$ be defined by \eqref{hSt:Def}. Then
\begin{equation}\label{hdoub}
			h^{**}(\boldxi):=(h^{*})^*(\boldxi)=\inf_{\boldV\in D^\bot}\int_Y M(y,\boldxi+\boldV(y))\dy
		\end{equation}
and in addition $h^*$ can be equivalently expressed as
	\begin{equation}\label{HAlter}
		h^*(\boldxi)=\inf_{\boldW\in D^{\bot\bot}_0}\int_Y M^*(y,\boldxi+\boldW(y))\dy,
	\end{equation}
where $D^{\bot \bot}_0:=\{\boldW\in D^{\bot \bot}: \;\int_Y\boldW(y)\dy =0\}.$
\begin{proof}
		 First, we show \eqref{hdoub}. Using the definition of $h^*$, see \eqref{hSt:Def}, and defining  a functional $\mathcal{F}:L^{M^*}(Y;\eR^{d\times N})\rightarrow\eR$ as
		\begin{equation*}
			\mathcal{F}(\boldW):=\int_Y M^*(y,\boldW(y))\dy
		\end{equation*}
we obtain (due to the fact that $\boldW\in D_0$ has zero mean value)
\begin{equation}\label{fStCom}
		\begin{split}
			h^{**}(\boldxi)&=\sup_{\boldeta\in\eR^{d\times N}}\left\{\boldxi\cdot\boldeta-\inf_{\boldW\in D_0}\mathcal{F}(\boldeta+\boldW)\right\}\\
&=\sup_{\boldeta\in\eR^{d\times N}}\left\{-\inf_{\boldW\in D_0}\left\{ \mathcal{F}(\boldeta+\boldW)-\int_{Y}\boldxi\cdot(\boldeta+\boldW(y))\dy\right\}\right\}\\
&=-\inf_{\boldeta\in\eR^{d\times N}}\left\{\inf_{\boldW\in D_0}\left\{\mathcal{F}(\boldeta+\boldW)-\int_{Y}\boldxi\cdot(\boldeta+\boldW(y))\dy\right\}\right\}\\
&=-\inf_{\boldV\in \eR^{d\times N}\oplus D_0}\left\{\mathcal{F}(\boldV)-\int_{Y}\boldxi\cdot\boldV(y)\dy\right\}.
			\end{split}
		\end{equation}
Next, we apply Lemma~\ref{Lem:Duality} onto a functional $\mathcal{F}$. First, we observe that $\mathcal{F}$ is closed, i.e., equivalently if $\boldW^k\rightarrow\boldW$ in $L^{M^*}(Y;\eR^{d\times N})$ then
		\begin{equation}\label{FClo}
			\liminf_{k\rightarrow\infty}\mathcal{F}(\boldW^k)\geq\mathcal{F}(\boldW).
		\end{equation}
		Obviously $\boldW^k\rightarrow\boldW$ in $L^{M^*}_{per}(Y;\eR^{d\times N})$ implies $\boldW^k\rightarrow\boldW$ in $L^1_{per}(Y;\eR^{d\times N})$. In order to show \eqref{FClo} it suffices to apply the lower semicontinuity of integral functionals with a Carath\'eodory integrand, see \cite[Theorem 4.2]{G03}. Moreover, $\mathcal{F}$ is continuous at $\bzero\in D$, which is a consequence of \eqref{SSS}. The conjugate functional $\mathcal{F}^*$ to $\mathcal{F}$ is given by
\begin{equation*}
			\mathcal{F}^*(\boldW^*)=\int_Y M(y,\boldW^*(y))\dy
\end{equation*}
according to \eqref{SST}. Therefore by Lemma \ref{Lem:Duality} we get from \eqref{fStCom}
\begin{equation*}
			h^{**}(\boldxi)=\inf_{\boldW^*\in (\eR^{d\times N}\oplus D_0)^\bot}\int_{Y}M(y,\boldW^*(y)+\boldxi)\dy.
\end{equation*}
Since,
		\begin{equation*}
			\left(\eR^{d\times N}\oplus D_0\right)^\bot=D^{\bot}
		\end{equation*}
we conclude \eqref{hdoub}. 		

Then we proceed with the computation of $h^{***}:=(h^{**})^*$. Since $M$ satisfies $\Delta_2$--condition, we have
		\begin{equation*}
			D^\bot=\{\boldV\in E^M_{per,\div}(Y;\eR^{d\times N}):\int_Y\boldV(y)\cdot\boldW(y)\dy=0\text{ for all }\boldW\in D\}
		\end{equation*}
		and $(\eR^{d\times N}\oplus D^\bot)^\bot=D^{\bot\bot}_0$. Accordingly, we obtain by Lemma~\ref{Lem:Duality}
		\begin{equation*}
			h^{***}(\boldxi)=-\inf_{\boldV\in\eR^{d\times N}\oplus D^{\bot}}\int_Y M(y,\boldV(y))-\boldxi\cdot\boldV(y)\dy=\inf_{\boldW\in D^{\bot\bot}_0}\int_Y M^*(y,\boldxi+\boldW(y))\dy.
		\end{equation*}
		As $h^*$ is convex and continuous, the latter identity implies \eqref{HAlter}.
	\end{proof}
\end{Lemma}

The ${\mathcal N}-$functions $f$ and $f^*$, $h^*$ and $h^{**}$ respectively, were introduced in order to indicate the growth and coercivity properties of the operator $\hat\boldA$ as it is stated among other properties of $\hat\boldA$ in the following lemma.
\begin{Lemma}
Let the operator $\boldA$ satisfy \ref{AO}--\ref{ATh} and the $\mathcal{N}$--function $M$ satisfy \ref{MO}--\ref{MTh}. Then:
\begin{enumerate}[label=(\^{A}\arabic*)]
	\item \label{AhO} There is a constant $c>0$ such that for all $\boldxi\in\eR^{d\times N}$
		\begin{equation}\label{AhEst}
			\begin{split}
				\hat\boldA(\boldxi)\cdot\boldxi&\geq c(f(\boldxi)+f^*(\hat\boldA(\boldxi)))\text{ provided that }M^* \text{ satisfies }\Delta_2\text{--condition},\\
				\hat\boldA(\boldxi)\cdot\boldxi&\geq c(h^{**}(\boldxi)+h^{*}(\hat\boldA(\boldxi)))\text{ provided that }M \text{ satisfies }\Delta_2\text{--condition}
			\end{split}
		\end{equation}
		
	\item \label{AhT} For all $\boldxi,\boldeta\in\eR^{d\times N}$, $\boldxi\neq\boldeta$
		\begin{equation*}
			(\hat\boldA(\boldxi)-\hat\boldA(\boldeta))\cdot(\boldxi-\boldeta)> 0.
		\end{equation*}
	\item\label{AhTh} $\hat\boldA$ is continuous on $\eR^{d\times N}$.
\end{enumerate}
\begin{proof}
Let $\boldw_{\boldxi}$ be a solution of the cell problem corresponding to $\boldxi\in\eR^{d\times N}$, see \eqref{CellPr}, whose existence and uniqueness is granted by Lemma~\ref{L2.3}.  In addition, if $M^{*}$ satisfies $\Delta_2$--condition, we know that $\boldw_{\boldxi}\in W^{1}_{per}L^M(Y)$ and if $M$ satisfies $\Delta_2$--condition then $\boldw_{\boldxi}\in V_{per}^M$. Furthermore, in both cases we know that
\begin{equation}\label{dret}
\int_Y \boldA(y,\boldxi+\nabla \boldw_{\boldxi}) \cdot \nabla \boldw_{\boldxi}\dy =0.
\end{equation}
Then it directly follows that
\begin{equation}\label{BelEst}
	\begin{split}
	\hat{\boldA}(\boldxi)\cdot\boldxi&=\int_Y\boldA(y,\boldxi+\nabla \boldw_{\boldxi}(y))\dy\cdot\boldxi=\int_Y\boldA(y,\boldxi+\nabla \boldw_{\boldxi}(y))\dy\cdot(\boldxi+\nabla \boldw_{\boldxi}(y))\dy\\
	&\geq c\int_Y M(y,\boldxi+\nabla \boldw_{\boldxi}(y))+M^*(y,\boldA(y,\boldxi+\nabla \boldw_{\boldxi}(y))\dy
	\end{split}
\end{equation}
and the estimate \eqref{AhEst} will be derived from the above inequality.

First, we deal with the case $M^*$ satisfies $\Delta_2$--condition. Let us show $\nabla\boldw_{\boldxi}\in G^{\bot\bot}$. Choosing an arbitrary $\boldV\in G^\bot\subset E^{M^*}(Y;\eR^{d\times N})$ and taking into account that $\boldw_{\boldxi}\in W^{1}_{per}L^M(Y;\eR^N)$ is a weak$^*$ limit of a sequence $\{\boldw_{\boldxi}^k\}_{k=1}^\infty\subset W^{1}_{per}E^M(Y;\eR^{N})$ we obtain
\begin{equation*}
	\int_Y \nabla\boldw_{\boldxi}(y)\cdot \boldV(y)\dy=\lim_{k\rightarrow \infty}\int_Y \nabla\boldw_{\boldxi}^k(y)\cdot\boldV(y)\dy=0.
\end{equation*}
As $\nabla\boldw_{\boldxi}\in G^{\bot\bot}$, we can use  Lemma~\ref{Lem:FEqExpr} to infer
\begin{equation}\label{EstF}
	\int_Y M(y,\boldxi+\nabla\boldw_{\boldxi}(y))\dy\geq f(\boldxi).
\end{equation}
As $\boldw_{\boldxi}$ is a solution of the cell problem and $G\subset \{\nabla\boldv:\boldv\in W^{1}_{per}L^M(Y;\eR^{N})\}$, it follows from weak formulation~\eqref{CellPrWFMStDTw} that $\boldA(\cdot,\boldxi+\nabla \boldw)\in G^\bot$. Thus we deduce
\begin{equation}\label{EstFSt}
	\int_Y M^*(y,\boldA(y,\boldxi+\nabla\boldw(y))\dy\geq f^*(\boldxi).
\end{equation}
Finally, estimate \eqref{AhEst}$_1$ is obtained as a consequence of \eqref{BelEst}, \eqref{EstF} and \eqref{EstFSt}.

Next, we assume that $M$ satisfies $\Delta_2$--condition. Since $D$ is defined as a closure of smooth periodic divergence-free function, then it directly follows that $\nabla\boldw_{\boldxi}\in D^\bot$, which implies
\begin{equation}\label{EstHDSt}
	\int_Y M(y,\boldxi+\nabla\boldw_{\boldxi}(y))\dy\geq h^{**}(\boldxi).
\end{equation}
Next, since  $\boldA(\cdot,\boldxi+\nabla\boldw_{\boldxi})\in L^{M^*}_{per}(Y;\eR^{d\times N})$ and  \eqref{DBotChar} is available (due to the fact that $M$ satisfies $\Delta_2$--condition), we conclude using weak formulation \eqref{CellPrWFMDTw} that $\boldA(\cdot,\boldxi+\nabla\boldw_{\boldxi})\in D^{\bot\bot}$. Consequently, by Lemma~\ref{Lem:HStEqExpr} we get
\begin{equation}\label{EstHSt}
	\int_Y M^*(y,\boldA(y,\boldxi+\nabla\boldw_{\boldxi}(y)))\dy\geq h^*(\hat\boldA(\boldxi)).
\end{equation}
Estimate \eqref{AhEst}$_2$ then follows from \eqref{BelEst}, \eqref{EstHSt} and \eqref{EstHDSt}.

In order to show \ref{AhT} we fix $\boldxi_1,\boldxi_2\in\eR^{d\times N}, \boldxi_1\neq\boldxi_2$ and find corresponding weak solutions of the cell problem $\boldw_1$ and $\boldw_2$. One obtains
\begin{equation*}
\int_Y \boldA(y,\boldxi_i+\nabla \boldw_i(y))\cdot\nabla \boldw_j(y)\dy=0\text{ for } i,j=1,2
\end{equation*}
in the same way as \eqref{TestCPrSol} was shown. Then using \ref{AF}, we deduce
\begin{equation*}
	\begin{split}
	(\hat\boldA(\boldxi_1)-\hat\boldA(\boldxi_2))\cdot(\boldxi_1-\boldxi_2)&=\int_Y(\boldA(y,\boldxi_1+\nabla \boldw_1)-\boldA(y,\boldxi_2+\nabla \boldw_2))\cdot(\boldxi_1-\boldxi_2)\dy\\
	&=\int_Y(\boldA(y,\boldxi_1+\nabla \boldw_1)-\boldA(y,\boldxi_2+\nabla \boldw_2))\cdot(\boldxi_1+\nabla \boldw_1-\boldxi_2-\nabla \boldw_2)\dy>0
	\end{split}
\end{equation*}
and \ref{AhT} follows.

To show \ref{AhTh} we consider $\{\boldxi^k\}_{k=1}^\infty$ such that $\boldxi^k\rightarrow\boldxi$ in $\eR^{d\times N}$ as $k\rightarrow\infty$, a corresponding sequence of weak solutions of the cell problems $\{\boldw^k\}_{k=1}^\infty$ and $\boldw$ corresponding to $\boldxi$. Then we have for an arbitrary but fixed $\boldeta\in \eR^{d\times N}$ that
\begin{equation*}
	(\hat\boldA(\boldxi^k)-\hat\boldA(\boldxi))\cdot\boldeta=\int_Y (\boldA(y,\boldxi^k+\nabla \boldw^k)-\boldA(y,\boldxi+\nabla \boldw))\cdot\boldeta\dy\rightarrow 0
\end{equation*}
as $k\rightarrow\infty$ by \eqref{WSCont}. Since $\eR^{d\times N}$ is finite dimensional, we conclude \ref{AhTh} from the latter convergence.
\end{proof}
\end{Lemma}


\section{Proof of Theorem~\ref{Thm:MainOne}} \label{SS3}
This section is devoted to the proof of Theorem~\ref{Thm:MainTwo}. Since the theorem is formulated in a slightly vague way without the precise definition of the function spaces and the notion of the weak solution, we formulate here two Lemmata, which cover the statement of Theorem~\ref{Thm:MainOne}.

First, we consider the case of $M^*$ satisfying $\Delta_2$--condition. Since the existence of a solution for this case can be proven following the approach from \cite{GSG208}, which in fact deals with the existence result for a more complex system governing the flow of non-Newtonian fluids, we do not prove it here but for the reader convenience present the proof in Appendix~\ref{Ape3}.
\begin{Lemma}\label{Lem:ExUnMStDTwo}
Let $N\geq 1$, $\Omega\subset\Rd$ be a bounded Lipschitz domain with $d\geq 2$. Assume that  an operator $\boldA$ satisfies \ref{AO}, \ref{ATh} and \ref{AF} and  $M:\Omega\times\eR^{d\times N}\rightarrow[0,\infty)$ is an ${\mathcal N}$--function such that $M^*$ satisfies $\Delta_2$--condition and for all $R>0$ we have
\begin{equation}\label{IntMSphere2}
	\int_\Omega\sup_{|\boldxi|=R}M^*(x,\boldxi)\dx<\infty.
\end{equation}
Then there exists a unique weak solution to problem \eqref{EllPr}, which is a function $\boldu\in W^1_0L^M(\Omega;\eR^N)$ such that
\begin{equation}\label{EllPrWeakFormMStDTw}
	\int_\Omega \boldA(x,\nabla\boldu(x))\cdot\nabla\bphi(x)\dx=\int_\Omega \boldF(x)\cdot\nabla\bphi(x)\dx
\end{equation}
is satisfied for all $\bphi\in W^1_0L^M(\Omega;\eR^N)$.
\end{Lemma}

The second part of the statement of Theorem~\ref{Thm:MainOne}, i.e., the case when $M$ satisfies $\Delta_2$--condition, is covered by the following lemma. Since this result is indeed new, we provide the complete proof here.
\begin{Lemma}\label{Lem:ExMDTwo}
Let $N\geq 1$, $\Omega\subset\Rd$ be a bounded Lipschitz domain with $d\geq 2$. Assume that  an operator $\boldA$ satisfies \ref{AO}, \ref{ATh} and \ref{AF} and  $M:\Omega\times\eR^{d\times N}\rightarrow[0,\infty)$ is an ${\mathcal N}$--function such that it satisfies  $\Delta_2$--condition and for all $R>0$ we have
\begin{equation}\label{IntMSphere}
	\int_\Omega\sup_{|\boldxi|=R}M(x,\boldxi)\dx<\infty.
\end{equation}
Then there exists a unique weak solution to problem \eqref{EllPr}, which is a function $\boldu\in V^M_0$ such that
\begin{equation}\label{EllPrWeakForm}
	\int_\Omega \boldA(x,\nabla\boldu(x))\cdot\nabla\bphi(x)\dx=\int_\Omega \boldF(x)\cdot\nabla\bphi(x)\dx
\end{equation}
is satisfied for all $\bphi\in V^M_0$.
\begin{proof}
We are not able to show directly the existence of $\boldu$ satisfying \eqref{EllPrWeakForm} because we have $\boldA(\cdot,\nabla \boldu)\in L^{M^*}(\Omega;\eR^{d\times N})$ only. Therefore we cannot utilize the method involving the weak$^*$ convergence in $L^M(\Omega)$. It turns out that one can find a weak solution of the dual problem to \eqref{EllPrWeakForm}, from which we then deduce the existence of a weak solution to the original problem. As $\boldA$ is strictly monotone  it is a homeomorphism on $\eR^{d\times N}$ therefore the inverse operator to $\boldA$, denoted as $\boldB$, exists. It is not difficult to show, thanks to \ref{ATh}--\ref{AF}, that the operator $\boldB$ fulfills
	\begin{equation}\label{MDTwoBGrCoercEst}
		\boldB(x,\boldzeta)\cdot\boldzeta\geq c\left(M(x,\boldB(x,\boldzeta))+M^*(x,\boldzeta)\right).
	\end{equation}
Moreover, $\boldB$ is strictly monotone.

Our first goal is to find  a function $\boldT\in L^{M^*}_{\div}\left(\Omega;\eR^{d\times N}\right)$ satisfying
	\begin{equation}\label{DualPr}
		\int_\Omega \boldB(x,\boldT(x)+\boldF(x))\cdot\boldW(x)\dx=0\text{ for all }\boldW\in E^{M^*}_{\div}(\Omega;\eR^{d\times N}).
	\end{equation}

The solvability of \eqref{DualPr} (the dual problem) will directly imply the statement of the lemma. Hence, let us focus on \eqref{DualPr}.
	We observe that the space $E^{M^*}_{\div}(\Omega;\eR^{d\times N})$ is separable since it is a closed subspace of the separable space $E^{M^*}\left(\Omega;\eR^{d\times N}\right)$. Thus there is $\{\boldW^i\}_{i=1}^\infty$, a linearly independent subset of $E^{M^*}_{\div}\left(\Omega;\eR^{d\times N}\right)$ such that $\overline{\bigcup_{k=1}^\infty Lin\{\boldW^i\}_{i=1}^k}^{\|\cdot\|_{L^{M}}}=E^{M^*}_{\div}\left(\Omega;\eR^{d\times N}\right)$. We construct Galerkin approximations to \eqref{DualPr}. We define $\boldT^k:=\sum_{i=1}^k\alpha_i^k\boldW^i$ for $k\in\eN$, where $\alpha^k_i\in\eR$ are chosen in such a way that
	\begin{equation}\label{GalAppDualFormMDTwo}
		\int_\Omega \boldB\left(x,\boldT^k+\boldF\right)\cdot\boldW^i\dx=0
	\end{equation}
	for all $i=1,\ldots,k$.\\
	\textbf{Step 1}: Let us show the existence of $(\alpha^{k}_1,\ldots,\alpha_k^k)\in\eR^k$ satisfying \eqref{GalAppDualFormMDTwo}. We want to apply Lemma \ref{Lem:FixP} on a mapping $\bolds:\eR^k\rightarrow \eR^k$ defined as
	\begin{equation*}
		s_j(\balpha)=\int_\Omega\boldB(x,\boldT^k+\boldF)\cdot\boldW^j\dx,\quad j=1,\ldots,k.
	\end{equation*}
	First, we show that $\bolds$ is continuous. Let us suppose that $\balpha^n\rightarrow\balpha$ in $\eR^k$. We observe that for $h^n_j:=\left(\boldB\left(x,\sum_{i=1}^k\alpha^n_i\boldW^i+\boldF\right)-\boldB\left(x,\sum_{i=1}^k\alpha_i\boldW^i\right)\right)\cdot\boldW^j$ we have $h^n_j\rightarrow 0$ as $n\rightarrow\infty$ a.e. in $\Omega$. Then, using \eqref{MDTwoBGrCoercEst} and the Young inequality one derives the estimate
	\begin{equation*}
		\int_\Omega M\left(x,\boldB\left(x,\sum_{i=1}^k\alpha^n_i\boldW^i\right)\right)\dx\leq c.
	\end{equation*}
	Therefore $|h^n_j|$ is uniformly integrable since we also have $\boldB\left(x,\sum_{i=1}^k\alpha_i\boldW^i\right)\in L^1\left(\Omega;\eR^{d\times N}\right)$ and $\boldW_j\in L^\infty\left(\Omega;\eR^{d\times N}\right)$. Consequently, $\bolds$ is continuous since by the Vitali theorem we get
	\begin{equation*}
		|\bolds(\balpha^n)-\bolds(\balpha)|\leq k\max_{j=1,\ldots,k}\int_\Omega |h^n_j|\dx\rightarrow 0\text{ as }n\rightarrow\infty.
	\end{equation*}
	Next, we verify that $\bolds$ satisfies \eqref{NonNegCond}. We denote $\boldW(\balpha):=\sum_{i=1}^k\alpha_i\boldW^i$ and show that
	\begin{equation}\label{MDTwAuxCon}
		\|\boldW(\balpha)+\boldF\|_{L^{M^*}(\Omega)}\rightarrow\infty\text{ as }|\balpha|\rightarrow \infty.
	\end{equation}
	We observe that $\min_{|\balpha|=1}\|\boldW(\balpha)\|>0$, this follows from the fact that $\{\boldW^i\}_{i=1}^k$ are linearly independent. Since $\boldF\in L^\infty\left(\Omega;\eR^{d\times N}\right)$, we find $R_0>0$ such that $\frac{\|\boldF\|_{L^{M^*}}}{|\balpha|}\leq \frac{1}{2}\min_{|\bbeta|=1}\left\|\boldW(\bbeta)\right\|_{L^{M^*}(\Omega)}$ for all $\balpha\in\eR^k$ with $|\balpha|\geq R_0$. Considering such $\balpha$ we get by the triangle inequality
	\begin{equation*}
		\begin{split}
		\|\boldW(\balpha)+\boldF\|_{L^{M^*}(\Omega)}&\geq \|\boldW(\balpha)\|_{L^{M^*}(\Omega)}-\|\boldF\|_{L^{M^*}(\Omega)}\geq |\balpha|\left(\left\|\boldW\left(\frac{\alpha}{|\balpha|}\right)\right\|_{L^{M^*}(\Omega)}-\frac{\|\boldF\|_{L^{M^*}(\Omega)}}{|\balpha|}\right)\\
		&\geq \frac{1}{2}|\balpha|\min_{|\bbeta|=1}\|\boldW(\bbeta))\|_{L^{M^*}(\Omega)}.
		\end{split}
	\end{equation*}
	Hence \eqref{MDTwAuxCon} follows. By \eqref{MDTwoBGrCoercEst}, the Young inequality and \eqref{SSF} we have
	\begin{equation*}
		\begin{split}	
		\bolds(\balpha)\cdot\balpha&=\int_\Omega \boldB\left(x,\boldW(\balpha)+\boldF\right)\cdot\boldW(\balpha)\dx\geq \int_\Omega cM^*\left(x,\boldW(\balpha)+\boldF\right)-M^*\left(x,\frac{1}{c}\boldF\right)\dx\\
		&=c\int_\Omega M^*\left(x,\boldW{\balpha}+\boldF\right)\dx-\int_\Omega M^*\left(x,\frac{1}{c}\boldF\right)\dx\\
		&\geq c\|\boldW(\balpha)+\boldF\|_{L^{M^*}(\Omega)}-1-\int_\Omega M^*\left(x,\frac{1}{c}\boldF\right)\dx.
		\end{split}
	\end{equation*}
	Then using \eqref{MDTwAuxCon} we find $R\geq R_0$ such that $\bolds(\balpha)\cdot\balpha\geq 0$ for all $\balpha$ such that $|\balpha|=R$. Thus according to Lemma \ref{Lem:FixP} we have the existence of $\balpha$ satisfying \eqref{GalAppDualFormMDTwo}.
	\\
	\textbf{Step 2}: Multiplying \eqref{GalAppDualFormMDTwo} by $\alpha^k_i$ and summing over $i=1,\ldots,k$ yields
	\begin{equation}\label{MDTwoGalTestApp}
		\int_{\Omega}\boldB\left(x,\boldT^k+\boldF\right)\cdot\boldT^k\dx=0.
	\end{equation}
Applying \eqref{MDTwoBGrCoercEst}, the Young inequality and the convexity of $M$ in the second variable we deduce from \eqref{MDTwoGalTestApp} that
\begin{equation*}
\begin{split}
	c&\int_\Omega M\left(x,\boldB\left(x,\boldT^k+\boldF\right)\right)+M^*\left(x,\boldT^k+\boldF\right)\dx\leq  \int_\Omega \boldB\left(x,\boldT^k+\boldF\right)\cdot\left(\boldT^k+\boldF\right)\dx\\	
&=  \int_\Omega \frac{c}{2}\boldB\left(x,\boldT^k+\boldF\right)\cdot\frac{2}{c}\boldF\dx\leq \int_\Omega \frac{c}{2}M\left(x,\boldB\left(x,\boldT^k+\boldF\right)\right)+M^*\left(x,\frac{2}{c}\boldF\right)\dx.
\end{split}
\end{equation*}
Hence, it follows that
	\begin{equation}\label{MDTwApEst}
		\frac{c}{2}\int_\Omega M\left(x,\boldB\left(x,\boldT^k+\boldF\right)\right)\dx+c\int_\Omega M^*\left(x,\boldT^k+\boldF\right)\dx\leq \int_\Omega M^*\left(x,\frac{2}{c}\boldF\right)\dx.
	\end{equation}
Since the right hand side of the latter inequality is finite as $\boldF\in L^\infty(\Omega;\eR^{d\times N})$, we infer the existence of $\boldT\in L^{M^*}_{\div}(\Omega;\eR^{N})$ and $\bar\boldB\in E^{M}(\Omega;\eR^{d\times N})$ such that (note here that for \eqref{MDTwoDualPrAppConv}$_2$ we use the fact that $M$ satisfies $\Delta_2$--condition)
		\begin{equation}\label{MDTwoDualPrAppConv}
			\begin{alignedat}{2}
				\boldT^k+\boldF&\WSCon\boldT+\boldF&&\text{ in }L^{M^*}(\Omega;\eR^{d\times N}),\\
				\boldB(\cdot,\boldT^k+\boldF)&\WSCon \bar\boldB &&\text{ in }E^M(\Omega;\eR^{d\times N}).
			\end{alignedat}
		\end{equation}
		Employing the convergence \eqref{MDTwoDualPrAppConv}$_2$ in \eqref{MDTwoGalTestApp} we have for all $i\in\eN$
	\begin{equation}
		\int_\Omega \bar\boldB\cdot\boldW^i \dx =0.\label{end1}
	\end{equation}
Consequently, since $\{\boldW^i\}_{i=1}^{\infty}$ forms a basis we also have for all $\boldW\in E^{M^*}_{\div}\left(\Omega;\eR^{d\times N}\right)$
\begin{equation}
		\int_\Omega \bar\boldB\cdot\boldW \dx =0. \label{remains}
	\end{equation}
Thus to prove \eqref{DualPr}, it remains  to identify $\bar\boldW$.

Multiplying the $i$-th equation in \eqref{end1} by $\alpha_i^k$ and summing the result over $i=1,\ldots,k$ yields
	\begin{equation*}
		\int_\Omega \bar\boldB\cdot(\boldT^k+\boldF)\dx=\int_\Omega \bar\boldB\cdot\boldF\dx.
	\end{equation*}
	We apply the convergence \eqref{MDTwoDualPrAppConv}$_1$, which is possible since $\bar\boldB\in L^M(\Omega;\eR^{d\times N})=E^M(\Omega;\eR^{d\times N})$ as $M$ is assumed to satisfy $\Delta_2$--condition, to obtain
	\begin{equation}\label{MDTwoLimitProd}
		\int_\Omega \bar\boldB\cdot\boldT\dx=0.
	\end{equation}
	Let us identify $\bar\boldB$ with the help of the variant of Minty's trick for nonseparable and nonreflexive function spaces. First, using the monotonicity of $\boldB$ and \eqref{MDTwoGalTestApp} we get
	\begin{equation*}
	\begin{split}
		0&\leq \int_\Omega\left(\boldB(x,\boldT^k+\boldF)-\boldB(x,\boldW)\right)\cdot(\boldT^k+\boldF-\boldW)\dx\\&=\int_\Omega \boldB(x,\boldT^k+\boldF)\cdot(\boldF-\boldW)-\boldB(x,\boldW)\cdot(\boldT^k+\boldF-\boldW)\dx
	\end{split}
	\end{equation*}
	for an arbitrary but fixed $\boldW\in L^\infty(\Omega;\eR^{d\times N})$. Then performing the limit passage $k\to\infty$ in the latter inequality and using \eqref{MDTwoDualPrAppConv} and \eqref{MDTwoLimitProd} we arrive at
	\begin{equation}\label{LimMonoIneq}
		0\leq \int_\Omega \bar\boldB\cdot(\boldF-\boldW)-\boldB(x,\boldW)\cdot(\boldT+\boldF-\boldW)\dx=\int_\Omega \left(\bar\boldB-\boldB(x,\boldW)\right)\cdot(\boldT+\boldF-\boldW)\dx.
	\end{equation}
	Next, we denote for a positive $l$
\begin{equation*}
	\Omega_l=\{x\in\Omega: |\boldT(x)+\boldF(x)|\leq l\}
\end{equation*}
and let $\chi_l$ be the characteristic function of $\Omega_l$. Then choosing arbitrary $0<l<m$, $h\in(0,1)$ and $\boldZ\in L^\infty(\Omega;\eR^{d\times N})$ we set $\boldW=(\boldT+\boldF)\chi_m+h\boldZ\chi_l$ in \eqref{LimMonoIneq} to obtain
\begin{equation}\label{AlmostMintyIneq}
		0\leq \int_{\Omega\setminus\Omega_m} \bar\boldB\cdot(\boldT+\boldF)\dx -h\int_{\Omega_l}\left(\bar\boldB-\boldB(x,\boldT+\boldF+h\boldZ)\right)\cdot\boldZ\dx.
	\end{equation}
We see that $|\Omega\setminus\Omega_m|\rightarrow 0$ and $\bar\boldB\cdot(\boldT+\boldF)\chi_{\Omega\setminus\Omega_m}\to 0$ a.e. in $\Omega$ as $m\to\infty$. Hence by the Lebesgue dominated convergence theorem we get performing the limit passage $m\to\infty$ in \eqref{AlmostMintyIneq}
\begin{equation}\label{AlmostAlmostMintyIneq}
	0\leq -h\int_{\Omega_l}\left(\bar\boldB-\boldB(x,\boldT+\boldF+h\boldZ)\right)\cdot\boldZ\dx.
\end{equation}
We observe next that using \eqref{IntMSphere} we have
\begin{equation*}
	\int_{\Omega_l} M(x,\boldB(\boldT+\boldF+h\boldZ))\dx\leq \int_\Omega\sup_{|\boldxi|=\|\boldB(\boldT+\boldF+h\boldZ)\|_{L^\infty(\Omega_l)}}M(x,\boldxi)\dx<\infty
\end{equation*}
uniformly in $h\in(0,1)$ as $\sup_{h\in(0,1)}\|\boldB(\boldT+\boldF+h\boldZ)\|_{L^\infty(\Omega_l)}<\infty$, which follows from Lemma~\ref{Lem:UBSeq}. Then it follows that $\{\boldB(\cdot,\boldT+\boldF+h\boldZ)\}_{h\in(0,1)}$ is uniformly integrable by Lemma~\ref{Lem:UnifIntegr}. Furthermore, $\boldT+\boldF+h\boldZ\to \boldT+\boldF$ a.e. in $\Omega$ as $h\to 0$.  Accordingly, dividing \eqref{AlmostAlmostMintyIneq} by $h$ we infer by the Vitali convergence theorem
\begin{equation}\label{MintyIneq}
	0\geq\int_{\Omega_l}\left(\bar\boldB-\boldB(x,\boldT+\boldF)\right)\cdot\boldZ\dx
\end{equation}	
for any $\boldZ\in L^\infty(\Omega;\eR^{d\times N})$. Setting
\begin{equation*}
\boldZ=\frac{\bar\boldB-\boldB(x,\boldT+\boldF)}{1+|\bar\boldB-\boldB(x,\boldT+\boldF)|}
\end{equation*}
in \eqref{MintyIneq} we deduce $\bar\boldB(x)=\boldB(x,\boldT(x)+\boldF(x))$ in $\Omega_l$ for an arbitrary $l>0$. As $|\Omega\setminus\Omega_l|\to 0$ as $l\to\infty$, we conclude $\bar\boldB(x)=\boldB(x,\boldT(x)+\boldF(x))$ for a.a. $x\in\Omega$.\\
\textbf{Step 5}: Based on \eqref{DualPr}, we show that there is $\boldu\in V^M_0$ such that $\nabla \boldu=\boldB(\cdot,\boldT+\boldF)$ and $\boldu$ satisfies \eqref{EllPrWeakForm}. We denote $\tilde{\boldB}(x):=\boldB(x,\boldT(x)+\boldF(x))$. Extending $\tilde{\boldB}$ by zero in $\Rd\setminus\Omega$ we obtain from \eqref{DualPr} that
	\begin{equation}\label{OrtCond}
		\int_{\Rd}\tilde\boldB(x)\cdot\boldW(x)\dx=0
	\end{equation}
	for all $\boldW\in C^\infty(\Rd;\eR^{d\times N})$ with $\div\boldW=0$ in $\Rd$. We fix a sequence $\{\delta_k\}_{k=1}^\infty$ such that $\delta_k\rightarrow 0$ as $k\rightarrow\infty$ and denote as $\boldB^k$ a mollification of $\tilde\boldB$ with a parameter $\delta_k$. Obviously, we have $\boldB^k\in C^\infty(\Rd;\eR^{d\times N})$ and
	\begin{equation}\label{BkStrConv}
		\boldB^k\rightarrow\boldB\text{ in }L^1(\Rd;\eR^{d\times N})\text{ as }k\rightarrow\infty.
	\end{equation}
	By the Fubini theorem we obtain
	\begin{equation*}
		\int_{\Rd}\boldB^k\cdot\boldW\dx=\int_{\Rd}(\tilde{\boldB}*\rho_{\delta_k})(x)\cdot \boldW(x)\dx=\int_{\Rd}\tilde{\boldB}(z)\cdot(\boldW*\rho_{\delta_k})(z)\dz
	\end{equation*}
	for any $\boldW\in C^\infty(\Rd;\eR^{d\times N})$. As mollification preserves solenoidality, we have
	\begin{equation*}
		\int_{\Rd}\boldB^k\cdot\boldW\dx=0
	\end{equation*}
	for any $\boldW\in C^\infty(\Rd;\eR^{d\times N})$ with $\div\boldW=0$ in $\Rd$ by \eqref{OrtCond}. Then de Rham theorem yields the existence of a distribution $\boldp^k$ such that $\nabla \boldp^k=\boldB^k$. Thus we infer $\boldp^k\in C^\infty(\Rd;\eR^N)$. We fix $R>0$ such that for all $k\in\eN$ $\supp\boldB^k$ is contained in the ball $B_R$. Then for all $k\in\eN$ $\boldp^k$ is equal to some constant $\bar \boldp^k$ in $\Rd\setminus\Omega$ and defining $\uk:=\boldp^k-\bar \boldp^k$ we obtain by the Poincar\'e inequality that
	\begin{equation*}
		\|\uk-\boldu^l\|_{L^1(B_R)}\leq c\|\nabla (\uk-\boldu^l)\|_{L^1(B_R)}=c\|\boldB^k-\boldB^l\|_{L^1(B_R)}.
	\end{equation*}
	The latter inequality implies that $\{\uk\}_{k=1}^\infty$ is a Cauchy sequence in $L^1(B_R)$ since the sequence $\{\boldB^k\}_{k=1}^\infty$ is a Cauchy sequence in $L^1(B_R)$ due to \eqref{BkStrConv}. Therefore $\{\uk\}_{k=1}^\infty$ possesses a limit $\boldu\in L^1(B_R)$. Moreover, as the sequence $\{\nabla \uk\}_{k=1}^\infty$ converges strongly in $L^1(B_R)$
 and for all $k\in\eN$ $\supp \uk\subset B_R$, we have $\boldu\in W^{1,1}_0(B_R)$. Finally, we observe that $\nabla \boldu=\tilde \boldB$. Hence $\boldu$ is equal to a constant in $\Rd\setminus\Omega$ and this constant is zero. Thus we have $\boldu\in W^{1,1}_0(\Omega;\eR^N)$ and as $\nabla \boldu=\tilde\boldB\in E^M(\Omega;\eR^{d\times N})$, i.e., $\boldu\in V_0^M$, by the definition of $\boldB$ we obtain
	\begin{equation*}
		\int_\Omega \boldA(x,\nabla \boldu)\cdot \nabla\bphi\dx=\int_\Omega \boldA(x,\boldB(\boldT+\boldF))\cdot\nabla\bphi\dx=\int_\Omega(\boldT+\boldF)\cdot\nabla\bphi\dx
	\end{equation*}
	for all $\bphi\in V_0^M$. As $\boldT\in L^{M^*}_{\div}\left(\Omega;\eR^{d\times N}\right)$, we see, thanks to the fact that $M$ satisfies $\Delta_2$ condition that\footnote{Indeed, since $\boldT\in L^{M^*}_{\div}\left(\Omega;\eR^{d\times N}\right)$, it can be approximated in the weak$^*$ topology by a sequence of divergence free functions belonging to $E^{M^*}_{\div}\left(\Omega;\eR^{d\times N}\right)$ and this approximative sequence can be again approximated by a sequence of smooth divergence free functions in strong topology  and consequently
$$
\int_{\Omega} \boldT \cdot \nabla \boldv \dx = 0
$$
for all $\boldv \in V_0^M.$
}
$\boldu$ satisfies \eqref{EllPrWeakForm}.\\
	\textbf{Step 6}: One easily obtains uniqueness of a weak solution. Supposing that $\boldu_1, \boldu_2$ are different weak solutions of \eqref{EllPrWeakForm}, we get after testing the difference of weak formulations for $\boldu_1$ and $\boldu_2$ by the difference $\boldu_1-\boldu_2$, which is a proper test function in \eqref{EllPrWeakForm} as $\boldu_1,\boldu_2\in V^M_0$, that
	\begin{equation*}
		\int_\Omega \left(\boldA(x,\nabla \boldu_1)-\boldA(x,\nabla \boldu_2)\right)\cdot (\nabla \boldu_1-\nabla \boldu_2)\dx=0.
	\end{equation*}
	Since $\boldA$ is strictly monotone, we have $\nabla (\boldu_1-\boldu_2)=0$ a.e. in $\Omega$ and the zero trace of $\boldu_1-\boldu_2$ on $\partial\Omega$ implies $\boldu_1=\boldu_2$ a.e. in $\Omega$.
\end{proof}
\end{Lemma}

\section{Proof of Theorem~\ref{Thm:MainTwo}}\label{Sec:1.2}

Firstly, we formulate the lemmata concerning the existence and uniqueness of a solution to problem \eqref{StudPr} for an arbitrary but fixed $\varepsilon>0$ that are direct consequence of Theorem~\ref{Thm:MainOne}, where \eqref{Cis1}, \eqref{Cis2} respectively follow from \ref{MTh}. For simplicity we denote $M^\varepsilon(x,\boldxi)=M\left(\frac{x}{\varepsilon},\boldxi\right)$ for fixed $\varepsilon$.
\begin{Lemma}\label{Lem:WSExist}
 Let $\Omega\subset\Rd$ be a bounded Lipschitz domain, the operator $\boldA$ satisfy \ref{AO}-\ref{AF}, the $\mathcal{N}$--function $M$ satisfy \ref{MO}-\ref{MTh} and the conjugate $\mathcal{N}$--function $M^*$ satisfy $\Delta_2$--condition. Then for fixed $\varepsilon\in(0,1)$ there exists a unique weak solution of the problem \eqref{StudPr}, which is a function $\ueps\in W_0^1L^{M^\varepsilon}\left(\Omega;\eR^N\right)$ such that
\begin{equation}\label{WFEpsPr}
	\int_{\Omega} \boldA\left(\frac{x}{\varepsilon},\nabla\ueps(x)\right)\cdot\nabla\bphi(x)\dx=\int_\Omega \boldF(x)\cdot\nabla\bphi(x)\dx
\end{equation}
is satisfied for all $\bphi\in W^1_0L^{M^\varepsilon}\left(\Omega;\eR^N\right)$.
\end{Lemma}

\begin{Lemma}\label{Lem:WSExistMDTw}
 Let $\Omega\subset\Rd$ be a bounded Lipschitz domain, the operator $\boldA$ satisfy \ref{AO}-\ref{AF}, the $\mathcal{N}$--function $M$ satisfy \ref{MO}-\ref{MTh} and $\Delta_2$--condition. Then for fixed $\varepsilon\in(0,1)$ there exists a unique weak solution of the problem \eqref{StudPr}, which is a function $\ueps\in V_0^{M^\varepsilon}$ such that
\begin{equation}\label{WFEpsPrMDTw}
	\int_{\Omega} \boldA\left(\frac{x}{\varepsilon},\nabla\ueps(x)\right)\cdot\nabla\bphi(x)\dx=\int_\Omega \boldF(x)\cdot\nabla\bphi(x)\dx
\end{equation}
is satisfied for all $\bphi\in V_0^{M^\varepsilon}$.
\end{Lemma}
Next we state the estimate that is uniform with respect to $\varepsilon$.
\begin{Lemma}\label{Lem:AprBound}
	Let the assumptions of Lemmas \ref{Lem:WSExist} or \ref{Lem:WSExistMDTw} be satisfied and $\ueps$ be a weak solution of the problem \eqref{StudPr}.
	Then we have
	\begin{equation}\label{AprEst}
		\sup_{0<\varepsilon<1}\int_\Omega M^\varepsilon\left(\frac{x}{\varepsilon},\nabla\ueps(x)\right)+(M^\varepsilon)^*\left(\frac{x}{\varepsilon},\boldA^\varepsilon(x)\right)\dx\leq c<\infty
	\end{equation}
	and $\{\Aeps\}_{0<\varepsilon< 1}$ is bounded in $L^{m_2^*}(\Omega;\eR^{d\times N})$ and $\{\ueps\}_{0<\varepsilon< 1}$ is bounded in $V_0^{m_1}$.
	\begin{proof}
	Setting $\bphi=\ueps$ in \eqref{EllPrWeakFormMStDTw}, \eqref{EllPrWeakForm} respectively, we obtain the following identity
		\begin{equation}\label{WeakFormTestSol}
			\int_{\Omega}\Aeps\cdot\nabla\ueps\dx=\int_{\Omega}\boldF\cdot\nabla\ueps\dx.
		\end{equation}
		Using the Young inequality, the convexity of $M$ and the fact that the constant $c\leq 1$, which is an obvious consequence of the Young inequality, it follows from \eqref{WeakFormTestSol} that
		\begin{equation*}
			c\int_{\Omega} M^\varepsilon(x,\nabla\ueps)+(M^\varepsilon)^*(x,\Aeps)\dx\leq \int_{\Omega}(M^\varepsilon)^*\left(x,\frac{2}{c}\boldF\right)+\frac{c}{2}M^\varepsilon(x,\nabla\ueps)\dx.
		\end{equation*}
		Consequently, employing \ref{MTh} we obtain
		\begin{equation}\label{AEIneq}
			c\int_{\Omega}\frac{1}{2} m_1(|\nabla\ueps|)+m_2^*(|\Aeps|)\dx\leq c\int_{\Omega}\frac{1}{2} M^\varepsilon\left(x,\nabla\ueps\right)+(M^\varepsilon)^*\left(x,\Aeps\right)\dx\leq \int_{\Omega}m_1^*\left(\frac{2}{c}|\boldF|\right)\dx.
		\end{equation}
		Due to \eqref{Assumption:F} the integral on the right hand side is finite. Hence estimate \eqref{AprEst} and boundedness of $\{\ueps\}$ and $\{\boldA^\varepsilon\}$ follows provided we use also the Poincar\'e inequality, see e.g. \cite[Section 2.4]{G79}.
	\end{proof}
\end{Lemma}
Based on the previous lemma we obtain the following convergence results.
\begin{Lemma}
Let the assumptions of Lemmas \ref{Lem:WSExist} and \ref{Lem:WSExistMDTw} be satisfied. Let $\{\varepsilon_j\}_{j=1}^\infty$ be such that $\varepsilon_j\rightarrow 0$ as $j\rightarrow\infty$ and $\{\boldu^{\varepsilon_j}\}_{j=1}^\infty$ be a sequence of weak solutions of \eqref{StudPr}.
	Then there is a subsequence $\{\varepsilon_{j_k}\}_{k=1}^\infty$, functions $\boldu\in V_0^{m_1}$, $\boldU\in L^{m_1}(\Omega\times Y;\eR^{d\times N})$, $\bar\boldA\in L^{m_2^*}(\Omega;\eR^{d\times N})$ and $\boldA^0\in L^{m_2^*}(\Omega\times Y;\eR^{d\times N})$ such that as $k\rightarrow\infty$ we have the following weak convergence result (the sequences are denoted by $k$ and not by $\varepsilon_{j_k}$ for simplicity)
	\begin{equation}\label{WSConv}
		\begin{alignedat}{2}
			\boldu^k&\WSCon \boldu&&\text{ in }L^{m_1}(\Omega;\eR^N),\\
			\nabla\boldu^k&\WSCon \nabla\boldu&&\text{ in }L^{m_1}(\Omega;\eR^{d\times N}),\\
			\boldA^{k}&\WSCon\bar\boldA&&\text{ in }L^{m_2^*}(\Omega;\eR^{d\times N})
		\end{alignedat}
	\end{equation}
	and
	\begin{equation}\label{WTSC}
		\begin{alignedat}{2}
		 (\nabla \uk)\circ S_k&\WSCon\nabla \boldu+\boldU &&\text{ in }L^{M_y}(\Omega\times Y;\eR^{d\times N}),\\
		 \boldAk\circ S_k&\WSCon\boldA^0 &&\text{ in }L^{M_y^*}(\Omega\times Y;\eR^{d\times N}).
		\end{alignedat}
	\end{equation}
	The limit functions $\boldA$ and $\boldA^0$ are related via
	\begin{equation}\label{AZAv}
	\bar\boldA=\int_Y\boldA^0\dy.
	\end{equation}
	Moreover, assume that either $N=1$ or the embedding $W^{1}_0 L^{m_1}\hookrightarrow L^{m_2}$ holds. If $M^*$ satisfies $\Delta_2$--condition then for a.a. $x\in\Omega$
	\begin{align}
		\boldU&(x,\cdot)\in G^{\bot\bot},\label{BUReg}\\
		\boldA^0&(x,\cdot)\in G^\bot,\label{AzDual}\\
		\boldu&\in V^f_0\label{LimRegU},\\
		\bar\boldA&\in L^{f^*}(\Omega;\eR^{d\times N}).\label{bAReg}
	\end{align}
	and if $M$ satisfies $\Delta_2$--condition then for a.a. $x\in\Omega$
		\begin{align}
		\boldU&(x,\cdot)\in D^{\bot},\label{BURegT}\\
		\boldA^0&(x,\cdot)\in D^{\bot\bot},\label{AzDualT}\\
		\boldu&\in V^{h^{**}}_0,\label{LimRegUT}\\
		\bar\boldA&\in L^{h^*}(\Omega;\eR^{d\times N}).\label{bARegT}
	\end{align}
	The function $\bar\boldA$ satisfies
	\begin{equation}\label{WFLim}
		\int_\Omega\bar\boldA\cdot\nabla\bphi=\int_\Omega \boldF\cdot\nabla\bphi
	\end{equation}
	for all $\bphi\in C^\infty_c\left(\Omega;\eR^N\right)$.
	\begin{proof}
	Let us mention that in the proof we shall make several selections of subsequences not necessarily stressing this fact.	The convergence results in \eqref{WSConv} follow directly from the uniform estimates from Lemma~\ref{Lem:AprBound}. As a consequence of \eqref{WSConv}$_{1,2}$ and Lemma~\ref{Lem:Facts2S}~\ref{F2SSi} we obtain the existence of a function $\boldU\in L^{m_1}(\Omega\times Y;\eR^{d\times N})$ and a subsequence $\{\varepsilon_k\}_{k=1}^\infty$ such that
	\begin{equation}\label{GrUkTSCon}
		\begin{alignedat}{2}
			\boldu^k&\WTSSCon\boldu&&\text{ in }L^{m_1}(\Omega\times Y;\eR^N),\\
			\nabla\boldu^k&\WTSSCon \nabla\boldu+\boldU &&\text{ in }L^{m_1}(\Omega\times Y;\eR^{d\times N})
		\end{alignedat}
	\end{equation}
	with $\boldU$ satisfying
	\begin{equation}\label{OrtProp}
		\int_Y\boldU(x,y)\cdot\bpsi(y)\dy=0\qquad \forall\bpsi\in C^\infty_{per}(Y;\eR^{d\times N}).
	\end{equation}
	By Lemma~\ref{Lem:Facts2S}~\ref{F2SFi} and Lemma~\ref{Lem:AprBound} we infer the existence of function $\boldA^0\in L^{m_2^*}(\Omega\times Y;\eR^{d\times N})$ such that
	\begin{equation}\label{AkTsCon}
		\boldA^k\WTSSCon \boldA^0\text{ in }L^{m_2^*}(\Omega\times Y;\eR^{d\times N}).
	\end{equation}
	Using convergence results \eqref{GrUkTSCon} and \eqref{AkTsCon}, the weak lower semicontinuity~\ref{F2SSev} from Lemma~\ref{Lem:Facts2S} we infer
	\begin{equation}\label{LimIneq}
		\int_\Omega\int_Y M(y,\nabla\boldu+\boldU)+M^*(y,\boldA^0)\dy\dx\leq\liminf_{k\rightarrow \infty} \int_\Omega\int_Y M\left(\frac{x}{\varepsilon_k},\nabla\boldu^k\right)+M^*\left(\frac{x}{\varepsilon_k},\boldA^k\right)\dy\dx<\infty.
	\end{equation}
	Next, by Lemma~\ref{Lem:Decomp} and Lemma~\ref{Lem:AprBound} we get
	\begin{equation*}
		\begin{split}
		\sup_{\varepsilon>0}\int_{\Omega}\int_Y M(y,\nabla \boldu^{\varepsilon}(S_\varepsilon(x,y)))\dy\dx&=\sup_{\varepsilon>0}\int_\Omega M\left(\frac{x}{\varepsilon},\nabla \boldu^\varepsilon(x)\right)\dx<\infty,\\
		\sup_{\varepsilon>0}\int_{\Omega}\int_Y M^*(y,\boldA(y,\nabla \boldu^\varepsilon(S_\varepsilon(x,y)))\dy\dx&=\sup_{\varepsilon>0}\int_\Omega M^*\left(\frac{x}{\varepsilon},\boldA\left(\frac{x}{\varepsilon},\nabla\boldu^\varepsilon(x)\right)\right)\dx<\infty.
		\end{split}
	\end{equation*}
	Accordingly, we obtain the existence of functions $\boldV\in L^{M_y}(\Omega\times Y;\eR^{d\times N})$ and $\tilde\boldA\in L^{M^*_y}(\Omega\times Y;\eR^{d\times Y})$ and a sequence $\varepsilon_k\rightarrow 0$ as $k\rightarrow \infty$ such that
	\begin{equation}\label{Conv2}
	\begin{alignedat}{2}
		\nabla\boldu^{k}\circ S_k&\WSCon \boldV&&\text{ in }L^{M_y}(\Omega\times Y;\eR^{d\times N}),\\
		\boldA^k\circ S_k&\WSCon\tilde\boldA&&\text{ in }L^{M^*_y}(\Omega\times Y;\eR^{d\times N})
	\end{alignedat}
	\end{equation}
	as $k\rightarrow \infty$. Hence in view of \eqref{LimIneq} we infer using \eqref{GrUkTSCon} and \eqref{AkTsCon} that $\boldV=\nabla\boldu+\boldU$, $\tilde\boldA=\boldA^0$, i.e., we have concluded \eqref{WTSC}.
	By Lemma~\ref{Lem:Facts2S}~\ref{F2SS} we get \eqref{AZAv}.
	In order to show \eqref{BUReg}--\eqref{bAReg} we distinguish separately the cases $N=1$ and the accomplishment of the embedding $W^{1}_0L^{m_1}\hookrightarrow L^{m_2}$. First, we deal with the case $N=1$. We recall that the truncation operator $T_h$ was introduced in Lemma~\ref{Lem:ModConvGrTrunc} in the appendix. Next, we realize that analogously to deriving the convergence \eqref{WTSC} the following convergences can be derived for any $h>0$
	\begin{equation}\label{GrTruncConv}
		\begin{alignedat}{2}
			T_h u^k\circ S_k&\WSCon T_h u&&\text{ in }L^\infty(\Omega\times Y),\\
		\nabla T_hu^k\circ S_k&\WSCon \nabla T_h u+\boldU^h&&\text{ in }L^{M_y}(\Omega\times Y;\eR^d)
		\end{alignedat}
	\end{equation}
	as $k\rightarrow\infty$. Note that the fact that $\{T_h u^k\}$ contains subsequence, which we will not relabel, converging weakly$^*$ in $L^\infty(\Omega\times Y)$ follows from the uniform estimate $\|T_h u^k\|_{L^\infty}\leq h$. The limit function is identified with the Lebesgue dominated convergence theorem from \eqref{WSConv}$_{1,2}$ and the compact embedding of $W^{1,1}(\Omega)$ to $L^1(\Omega)$.	Next, choose an arbitrary, but fixed $\varphi\in C^\infty_c(\Omega)$, $\boldV\in G^\bot$ and without loss of generality assume that $\int_Y \boldV=0$. Then, as $\div \boldV=0$ a.e. in $Y$ we obtain for an arbitrary but fixed $h>0$ using the integration by parts and Lemma~\ref{Lem:Decomp}
	\begin{align*}
	&\int_\Omega\int_Y \nabla T_hu^k(S_k(x,y))\cdot\boldV(y)\varphi(S_k(x,y))\dy\dx=
		\int_\Omega \nabla T_hu^k(x)\cdot\boldV(\frac{x}{\varepsilon_k})\varphi(x)\dx\\&=-\int_\Omega T_h u^k(x) \boldV(\frac{x}{\varepsilon})\cdot\nabla\varphi(x)\dx=-\int_\Omega\int_Y T_hu^k(S_k(x,y))\cdot\boldV(y)\cdot\nabla\varphi(S_k(x,y))\dy\dx.
	\end{align*}
	Performing the limit passage $k\rightarrow\infty$ in the latter identity with the help of \eqref{GrTruncConv} yields for an arbitrary but fixed $h>0$
	\begin{equation*}
		\begin{split}
		&\int_\Omega\int_Y \boldU^h(x,y)\cdot\boldV(y)\varphi(x)\dy\dx=\int_\Omega\int_Y (\nabla T_h u(x)+\boldU^h(x,y))\cdot\boldV(y)\varphi(x)\dy\dx\\&=-\left(\int_Y\boldV(y)\dy\right)\cdot\int_\Omega T_h u(x)\nabla\varphi(x)\dx=0,
		\end{split}
	\end{equation*}
	which means that $\boldU^h\in G^{\bot\bot}$ for any $h>0$. Denote $\boldW^h=\nabla T_h u+\boldU^h$ and $\boldW=\nabla u+\boldU$. Employing Lemma~\ref{Lem:Facts2S}~\ref{F2SSev} we infer
	\begin{equation*}
		\int_\Omega\int_Y M(y,\boldW^h(x,y))\dy\dx\leq \liminf_{k\to\infty}\int_\Omega M\left(\frac{x}{\varepsilon},\nabla T_h u^k(x)\right)\dx\leq \liminf_{k\to\infty}\int_\Omega M\left(\frac{x}{\varepsilon},\nabla u^k(x)\right)\dx<\infty,
	\end{equation*}
	which implies for $\boldW^j=\boldW^{h_j}$, where $h_j\to 0$ as $j\to\infty$, that $\boldW^j\WSCon \tilde\boldW$ in $L^{M_y}(\Omega\times Y;\eR^{d})$. On the other hand we obtain due to the weak lower semicontinuity of $L^1$--norm that
	\begin{equation*}
		\begin{split}
		\int_\Omega\int_Y |\boldW(x,y)-\boldW^j(x,y)|\dy\dx&\leq \liminf_{k\to\infty}\int_\Omega\int_Y |\nabla u^k(S_k(x,y))-\nabla T_{h_j}u^k(S_k(x,y))|\dy\dx\\
		&=\liminf_{k\to\infty}\int_{\{|u^k(S_k(x,y))|>h_j\}}|\nabla u^k(S_k(x,y,))|\dy\dx\\
		&\leq c\mu(|\{|u^k(S_k(x,y))|>h_j\}|),
		\end{split}
	\end{equation*}
	where $\mu$ is continuous at $0$ and $\mu(0)=0$. Thus we have $\boldW^j\to\boldW$ in $L^1(\Omega\times Y;\eR^d)$ due to the uniform bound on $\{u^k\circ S_k\}_{k=1}^\infty$ following from Lemma~\ref{Lem:AprBound}. Consequently, $\tilde\boldW=\boldW=\nabla u+\boldU$ a.e. in $\Omega\times Y$ and as this fact along with $\nabla T_{h_j}u\ModConvM\nabla u$ in $L^M(\Omega)$ implies $\boldU^j\ModConv{M^*}\boldU$ in $L^{M^*}(\Omega\times Y;\eR^d)$ and Lemma~\ref{Lem:ProdConv}, we obtain
	\begin{equation*}
		\int_\Omega\int_Y \boldU(x,y)\cdot\boldV(y)\varphi(x)\dy\dx=\lim_{j\to\infty}\int_\Omega\int_Y \boldU^j(x,y)\cdot\boldV(y)\varphi(x)\dy\dx=0,
	\end{equation*}
	from which \eqref{BUReg} follows.

	We obtain immediately for an arbitrary $\nabla\boldv=\boldV\in G$ and $\varphi\in C^\infty_c(\Omega)$ using Lemma~\ref{Lem:Decomp}, Lemma~\ref{Lem:ModConvGrTrunc}, weak formulation \eqref{WFEpsPr} and convergence \eqref{Conv2}$_2$
	\begin{equation*}
		\begin{split}
		&\int_\Omega\int_Y \boldA^0(x,y)\cdot\nabla\boldv(y)\varphi(x)\dy\dx\\
		&=\lim_{l\to\infty}\int_\Omega\int_Y \boldA^0(x,y)\cdot\nabla T_l\boldv(y)\varphi(x)\dy\dx=\lim_{l\to\infty}\lim_{k\to\infty}\int_\Omega\int_Y \boldA^k(S_k(x,y))\cdot\nabla T_l\boldv(y)\varphi(S_k(x,y))\dy\dx\\
		&=\lim_{l\to\infty}\lim_{k\to\infty}\varepsilon_k\int_\Omega \boldA^k(x)\cdot\nabla_x\left(T_l\boldv\left(\frac{x}{\varepsilon_k}\right)\varphi(x)\right)\dx-\varepsilon_k\int_\Omega \boldA^k(x)\cdot T_l\boldv\left(\frac{x}{\varepsilon_k}\right)\otimes\nabla\varphi(x)\dx\\&=\lim_{l\to\infty}\lim_{k\to\infty}\varepsilon_k\int_\Omega\boldF(x)\cdot\nabla_x\left(T_l\boldv\left(\frac{x}{\varepsilon_k}\right)\varphi(x)\right)\dx-\varepsilon_k\int_\Omega \boldA^k(x)\cdot T_l\boldv\left(\frac{x}{\varepsilon_k}\right)\otimes\nabla\varphi(x)\dx\\&=\lim_{l\to\infty}\lim_{k\to\infty}\int_\Omega\int_Y\boldF(S_k(x,y))\cdot\nabla_y\left(T_l\boldv(y)\varphi(S_k(x,y))\right)\dy\dx\\
		&\qquad-\varepsilon_k\int_\Omega\int_Y \boldA^k(S_k(x,y))\cdot T_l\boldv(y)\otimes\nabla\varphi(S_k(x,y))\dy\dx\\
		&=\lim_{l\to\infty}\lim_{k\to\infty}\int_\Omega\int_Y\boldF(S_k(x,y))\cdot\nabla_yT_l\boldv(y)\varphi(S_k(x,y))\dy\dx\\
		&\qquad+\varepsilon_k\int_\Omega\int_Y\boldF(S_k(x,y))\cdot T_l\boldv(y)\otimes\nabla\varphi(S_k(x,y))\dy\dx\\
		&\qquad-\varepsilon_k\int_\Omega\int_Y \boldA^k(S_k(x,y))\cdot T_l\boldv(y)\otimes\nabla\varphi(S_k(x,y))\dy\dx\\
		&=\lim_{l\to\infty}\int_\Omega\boldF(x)\varphi(x)\dx\cdot\int_Y\nabla_yT_l\boldv(y)\dy=0
		\end{split}
	\end{equation*}
	where we also used the fact that $T_l\boldv$ is $Y$--periodic. Thus we have \eqref{AzDual}.
	Now, we consider that $N>1$ and $W^1_0L^{m_1}\hookrightarrow L^{m_2}$ is available.	Let us choose arbitrary but fixed $\varphi\in C^\infty_c(\Omega)$, $\boldV\in G^\bot$ and without loss of generality assume that $\int_Y \boldV=0$. Then as $\div \boldV=0$ a.e. in $Y$ we obtain, using the integration by parts and Lemma~\ref{Lem:Decomp}
	\begin{align*}
		&\int_\Omega\int_Y \nabla\boldu^k(S_k(x,y))\cdot\boldV(y)\varphi(S_k(x,y))\dy\dx=\int_\Omega \nabla\boldu^k(x)\cdot\boldV\left(\frac{x}{\varepsilon_k}\right)\varphi(x)\dx\\&=-\int_\Omega \boldV\left(\frac{x}{\varepsilon_k}\right)\cdot\boldu^k(x)\otimes\nabla\varphi(x)\dx=-\int_\Omega\int_Y \boldV\left(y\right)\cdot\boldu^k(S_k(x,y))\otimes\nabla\varphi(S_k(x,y))\dx.
	\end{align*}
	Let us notice that the assumed embedding ensures that the integral on the right hand side is meaningful. Employing convergences \eqref{Conv2}$_1$ on the left hand side and \eqref{GrUkTSCon}$_1$ on the right hand side of the latter identity we arrive at
	\begin{equation*}
		\begin{split}
		&\int_\Omega\int_Y \boldU(x,y)\cdot\boldV(y)\varphi(x)\dy\dx=\int_\Omega\int_Y (\nabla\boldu(x)+\boldU(x,y))\cdot\boldV(y)\varphi(x)\dy\dx\\&=\int_Y\boldV(y)\dy \cdot \int_\Omega\boldu(x)\otimes\nabla\varphi(x)\dy\dx=0,
		\end{split}
	\end{equation*}	
	which concludes \eqref{BUReg}.
	
	In order to show \eqref{AzDual} we choose an arbitrary but fixed $\varphi\in C^\infty_c(\Omega)$, and $\nabla\boldv\in G$ and obtain
	\begin{equation*}
		\begin{split}
			\int_\Omega\boldA^k(x)\cdot\nabla\boldv\left(\frac{x}{\varepsilon_k}\right)\varphi(x)\dx=&\varepsilon_k\int_\Omega \boldA^k(x)\cdot\nabla_x\boldv\left(\frac{x}{\varepsilon_k}\right)\varphi(x)\dx=\varepsilon_k\int_\Omega \boldA^k(x)\cdot\nabla_y\left(\boldv\left(\frac{x}{\varepsilon_k}\right)\varphi(x)\right)\\&-\varepsilon_k\int_\Omega \boldA^k(x)\cdot\boldv\left(\frac{x}{\varepsilon_k}\right)\otimes\nabla\varphi(x)\dx.
		\end{split}
	\end{equation*}
Notice that the second integral on the right hand side is well defined due to the embedding $W^{1}_0L^{m_1}\hookrightarrow L^{m_2}$. Using Lemma~\ref{Lem:Decomp} and weak formulation \eqref{WFEpsPr} we infer
\begin{equation}\label{OrthAzPrep}
		\begin{split}
		&\int_\Omega\int_Y\boldA^{k}(S_k(x,y))\cdot\nabla\boldv(y)\varphi(S_k(x,y))\dy\dx=
			\int_\Omega\boldA^{k}(x)\cdot\nabla\boldv\left(\frac{x}{\varepsilon_k}\right)\varphi(x)\dx\\
			&=\varepsilon_k\int_\Omega\boldA^{k}(x)\cdot\nabla_x\left(\boldv\left(\frac{x}{\varepsilon_k}\right)\varphi(x)\right)\dx-\varepsilon_k\int_\Omega\boldA^k(x)\cdot\boldv\left(\frac{x}{\varepsilon_k}\right)\otimes\nabla\varphi(x)\dx\\
			&=\varepsilon_k\int_\Omega \boldF(x)\cdot\nabla_x\left(\boldv\left(\frac{x}{\varepsilon_k}\right)\varphi(x)\right)-\varepsilon_k\int_\Omega\int_Y \boldA^k(S_k(x,y))\cdot\boldv(y)\otimes\nabla\varphi(S_k(x,y))\dy\dx\\
			&=\int_\Omega\int_Y \boldF(S_k(x,y))\cdot\nabla_y\boldv(y)\varphi(S_k(x,y))\dy\dx+\varepsilon_k\int_\Omega\int_Y\boldF(S_k(x,y))\cdot\boldv(y)\otimes\nabla\varphi(S_k(x,y))\dy\dx\\
&\qquad-\varepsilon_k\int_\Omega\int_Y \boldA^k(S_k(x,y))\cdot\boldv(y)\otimes\nabla\varphi(S_k(x,y))\dy\dx\\
&=:I^{k,1}+I^{k,2}+I^{k,3}.
		\end{split}
	\end{equation}
Performing the limit passage $k\rightarrow\infty$ in the latter identity we realize that the terms on the right hand side vanish. Indeed, we have by Lemma~\ref{Lem:AprBound}, the embedding $W^1_0L^{m_1}\hookrightarrow L^{m_2}$ and \ref{MTh} that
\begin{align*}
&\lim_{k\to\infty}I^{k,1}= \int_\Omega \boldF(x)\varphi(x)\dx\cdot\int_Y\nabla\boldv(y)\dy=0,\\
&\lim_{k\to\infty}I^{k,2}\leq c\limsup_{k\to\infty}\varepsilon_k\|\boldF\|_{L^\infty(\Omega)}\|\nabla\boldv\|_{L^M(Y)}\|\nabla\varphi\|_{L^\infty(\Omega)}=0,\\
&\lim_{k\to\infty}I^{k,3}\leq \limsup_{k\to\infty}\varepsilon_k\|\boldA^k\|_{L^{m_2^*}(\Omega)}\|\nabla\boldv\|_{L^M(Y)}\|\nabla\varphi\|_{L^\infty(\Omega)}=0.
\end{align*}
Hence employing also \eqref{WTSC}$_2$ in \eqref{OrthAzPrep} we get
\begin{equation*}
	\int_\Omega\int_Y\boldA^0(x,y)\cdot\nabla\boldv(y)\varphi(x)\dy\dx=0,
\end{equation*}
which implies \eqref{AzDual}.
To conclude \eqref{LimRegU} we employ Lemma~\ref{Lem:FEqExpr} and \eqref{LimIneq}. Using the expression for $f^*$, \eqref{AzDual} and \eqref{LimIneq} we obtain \eqref{bAReg}.

Let us deal with the case when $M$ satisfies $\Delta_2$--condition. We note that in order to prove \eqref{BURegT} we fix $\boldV\in D$, $\varphi\in C^\infty_c(\Omega)$ and proceed analogously to the proof of \eqref{BUReg}. Taking into account \eqref{DBotChar} we fix $\nabla\boldv\in D^\bot$, $\varphi\in C^\infty_c(\Omega)$ and repeating the proof of \eqref{AzDual} we obtain \eqref{AzDualT}. To conclude \eqref{LimRegUT} we employ Lemma~\ref{Lem:HStEqExpr} and \eqref{LimIneq}. Using the expression for $h^*$, \eqref{AzDual} and \eqref{LimIneq} we obtain \eqref{bAReg}.

The identity \eqref{WFLim} is obtained by performing the limit passage $k\rightarrow \infty$ in \eqref{WFEpsPr} with $\varepsilon=\varepsilon_{j_k}$ for smooth compactly supported test functions using the convergence \eqref{WSConv}$_2$.

	\end{proof}
\end{Lemma}

\def\bpsi{\boldsymbol{\psi}}
\def\boldAEpsk{\boldA_{\varepsilon_k}}

The rest of this section is devoted to the identification of $\bar\boldA$ in \eqref{WFLim}. The proof is divided into five steps. \\
\textbf{Step 1}: We show the following identity
\begin{equation}\label{LimIntProd}
	\lim_{k\rightarrow\infty}\int_\Omega \boldA^k\cdot\nabla \uk\dx=\int_\Omega \bar\boldA\cdot\nabla \boldu\dx.
\end{equation}
Using identity \eqref{WFEpsPr} for $\varepsilon=\varepsilon_k$ with $\bphi=\boldu^k$ we get with the help of convergence \eqref{WSConv}$_2$
\begin{equation}\label{PartIdent}
	\lim_{k\rightarrow\infty}\int_\Omega \boldA^k\cdot\nabla \uk\dx=\lim_{k\to\infty}\int_\Omega \boldF\cdot\nabla \boldu^k\dx=\int_\Omega \boldF\cdot\nabla \boldu\dx.
\end{equation}
First, assume that $M^*$ satisfies $\Delta_2$--condition and $N=1$. In the case of $\Omega$ being a bounded Lipschitz domain, there exist a finite family of open sets $\{\Omega_i\}_{i=1}^K$ and a finite family of balls $\{B_i(x_i,r_i)\}_{i=1}^K$ such that each $\Omega_i$ is star--shaped with respect to the ball $B_i(x_i,r_i)$ and $\Omega=\bigcup_{i=1}^K\Omega_i$. Let $\{\theta_i\}_{i=1}^K$ be a partition of unity subordinated to $\{\Omega_i\}$, i.e., $\theta_i\in C^\infty_c(\Omega_i)$, $0\leq\theta_i\leq 1$ in $\Omega$ and $\sum_{i=1}^K\theta_i=1$ in $\Omega$. We consider for each $j\in\eN$ the truncation $T_j u$ and its decomposition in the form
\begin{equation*}
	T_ju(x)=\sum_{i=1}^KT_ju(x)\theta_i(x),\ x\in\Omega.
\end{equation*}
As $\nabla(T_ju\theta_i)=\nabla T_ju\theta_i+T_ju\nabla\theta_i\in L^f(\Omega_i;\eR^d)$ and $\supp u\theta_i\subset\Omega_i$ for each $i,j\in\eN$, we can adopt with minor modifications the procedure of constructing an approximating sequence applied in the proof of \cite[Lemma 3.1]{GSG11} on a function on a time--space domain to find $\{u^{n,j}_i\}\subset C^\infty_c(\Omega)$ such that $\nabla u^{n,j}_i\ModConv{f}\nabla (T_j u\theta_i)$ for each $i=1,\ldots, K$. Then we define $u^{n,j}=\sum_{i=1}^K u^{n,j}_i$ and obtain from \eqref{WFLim} using Lemmas~\ref{Lem:ProdConv} and~\ref{Lem:ModConvGrTrunc}
\begin{equation*}
\int_\Omega \bar\boldA\cdot\nabla u\dx=\lim_{j\to\infty}\lim_{n\to\infty}\int_\Omega \bar\boldA\cdot\nabla u^{n,j}\dx=\lim_{j\to\infty}\lim_{n\to\infty}\int_\Omega \boldF\cdot\nabla  u^{n,j}\dx=\int_\Omega \boldF\cdot\nabla u\dx.
\end{equation*}
Hence \eqref{LimIntProd} follows from \eqref{PartIdent} and the latter identity.

Next, we assume that $N\geq 1$ and the embedding $W^{1}_0L^{m_1}\hookrightarrow L^{m_2}$ holds. Let us consider the following decomposition of $\boldu$
\begin{equation*}
	\boldu(x)=\sum_{i=1}^K\boldu(x)\theta_i(x),\ x\in\Omega
\end{equation*}
where $\{\theta_i\}_{i=1}^K$ is the partition of unity introduced above. Obviously, thanks to the assumed embedding and~\ref{MTh} $\nabla(\boldu\theta_i)\in L^f(\Omega_i;\eR^{d\times N})$ and $\supp \boldu\theta_i\subset\Omega_i$ for each $i=1,\ldots,K$. Modifying again the procedure from the proof of \cite[Lemma 3.1]{GSG11} we find sequences $\{\boldu^n_i\}_{n=1}^\infty\subset C^\infty_c(\Omega;\eR^{d\times N})$ such that $\nabla\boldu^n_i\ModConv{f}\nabla(\boldu\theta_i)$. Then by Lemma~\ref{Lem:ProdConv} we get for the sequence $\{\boldu^n\}_{n=1}^\infty$ defined by $\boldu^n=\sum_{i=1}^K\boldu^n_i$ for each $n\in\eN$
\begin{equation*}
\int_\Omega \bar\boldA\cdot\nabla \boldu\dx=\lim_{n\to\infty}\int_\Omega \bar\boldA\cdot\nabla \boldu^{n}\dx=\lim_{n\to\infty}\int_\Omega \boldF\cdot\nabla  \boldu^{n}\dx=\int_\Omega \boldF\cdot\nabla \boldu\dx,
\end{equation*}
which implies \eqref{LimIntProd} along with \eqref{PartIdent}.
We note that if $M$ satisfies $\Delta_2$--condition we proceed analogously using \eqref{BURegT}--\eqref{bARegT} and the approximation by smooth compactly supported functions in modular topology of gradients in $L^{h^{**}}(\Omega;\eR^d)$, $L^{h^{**}}(\Omega;\eR^{d\times N})$ respectively.\\
The next three steps follow the same lines as an analogue part of the proof in~\cite{BGKSG17}, however for completeness we include the main reasoning.

\noindent
\textbf{Step 2}:   In this  part we concentrate on showing  that the following inequality
 \begin{equation}\label{LimMonIneq}
	0\leq\int_\Omega\int_Y (\boldA^0(x,y)-\boldA(y,\boldV(x,y)))\cdot(\nabla \boldu(x)+\boldU(x,y)-\boldV(x,y))\dy\dx
\end{equation}
 holds for all $\boldV\in C^\infty_c(\Omega;C^\infty_{per}(Y;\eR^{d\times N}))$.
 By Lemma~\ref{Lem:UBSeq}  for  $\boldV\in C^\infty_c(\Omega;C^\infty_{per}(Y;\eR^{d\times N}))$ we have
  that $\boldA(\cdot,\boldV)\in L^\infty(\Omega\times Y;\eR^{d\times N})$. Due to the appropriate embeddings $L^\infty(\Omega\times Y;\eR^{d\times N})\subset E^{m^*_1}(\Omega\times Y;\eR^{d\times N})\subset E^{m^*_2}(\Omega\times Y;\eR^{d\times N})$ for $\boldV^k(x)=\boldV(x,x\varepsilon_k^{-1})$ and $\tilde\boldA^k(x):=\boldA(x\varepsilon_k^{-1},\boldV^k)$ we obtain for $k\to\infty$ that
\begin{equation}\label{S2SC}
	\begin{alignedat}{2}
		\boldV^k&\STSCon \boldV &&\text{ in }E^{m_i}(\Omega\times Y;\eR^{d\times N}),\\
		\tilde\boldA^k&\STSCon \boldA(\cdot,\boldV(\cdot,\cdot)) &&\text{ in }E^{m_i^*}(\Omega\times Y;\eR^{d\times N}), i=1,2,
	\end{alignedat}
\end{equation}
and consequently
\begin{equation}\label{LimProd}
\lim_{k\rightarrow \infty}\int_\Omega \tilde\boldA^k(x)\cdot\boldV^k(x)\dx=\int_\Omega\int_Y \boldA(y,\bpsi(x,y))\cdot\boldV(x,y)\dy\dx.
\end{equation}
By  \ref{AF} we have
\begin{equation}\label{Ineq-k}
	\begin{split}
 \int_\Omega (\boldA^k(x)-\tilde\boldA^k(x))\cdot(\nabla\uk(x)-\boldV^k(x))\dx\ge0.
 	\end{split}
\end{equation}
We shall conclude \eqref{LimMonIneq} by passing with $k\to\infty$ in~\eqref{Ineq-k}.
Using directly \eqref{LimIntProd} together with~\eqref{AZAv} yields  that
 \begin{equation}\label{I}
 \lim_{k\rightarrow \infty}\int_\Omega \boldA^k(x)\cdot\nabla\uk(x)\dx=\int_\Omega \int_Y \boldA^0 \cdot \nabla \boldu\\
	= \int_\Omega \int_Y \boldA^0 \cdot (\nabla \boldu+\boldU).
\end{equation}
Note that the last equality trivially follows from \eqref{BUReg} and \eqref{AzDual}.
To pass to the limit in the remaining terms we use \eqref{WTSC}, \eqref{S2SC} together with Lemma~\ref{Lem:Facts2S}~\ref{F2SSev} and \eqref{LimProd}. Thus the proof of this part is complete.
%

\noindent
\textbf{Step 3}:
Our  goal is to show that~\eqref{LimMonIneq} holds not only for $\boldV\in C^\infty_c(\Omega;C^\infty_{per}(Y;\eR^{d\times N}))$ but also for  $\boldV\in L^\infty(\Omega\times Y;\eR^{d\times N})$. For this purpose we take an arbitrary function $\boldV \in C^\infty_c(\Omega;C^\infty_{per}(Y;\eR^{d\times N}))$ and  consider a sequence $\{K^m\}_{m=1}^\infty$ of compact subsets of $\Omega$ such that $K^1\subset K^2\subset\ldots\Omega$ and $\bigcup_{m=1}^\infty K^m=\Omega$. Since  $\boldV^m:=\boldV\chi_{K^m}$ are bounded in $L^\infty(\Omega\times Y)$ for every $m\in\eN$, thus there exists
 a positive constant $c$ such that
\begin{equation}\label{AVMUnifBound}
	\|\boldA(\cdot, \boldV^m)\|_{L^\infty(\Omega\times Y)}\leq c \ \text{ for all }m\in\eN,
\end{equation}
see Lemma~\ref{Lem:UBSeq} for details.
Using~\ref{MTh} and~\eqref{AVMUnifBound} gives
\begin{equation*}
\begin{split}
	\int_\Omega\int_Y& M(y,\boldV^m)+M^*(y,\boldA(y,\boldV^m)\dy\dx\leq \int_\Omega\int_Y m_2(|\boldV^m|)+m_1^*(|\boldA(y,\boldV^m)|)\dy\dx\\
	&\leq \int_\Omega\int_Y m_2(\|\boldV^m\|_{L^\infty(\Omega\times Y)})+m_1^*(\|\boldA(\cdot,\boldV^m)\|_{L^\infty(\Omega\times Y)})\leq c.
\end{split}
\end{equation*}
With the help of Lemma~\ref{Lem:UnifIntegr} the above estimate allows to conclude that  $\{\boldV^m\}_{m=1}^\infty$ and $\{\boldA(\cdot,\boldV^m)\}_{m=1}^\infty$ are uniformly integrable.
The uniform integrability together with a convergence in measure of these sequences with a use of Lemma~\ref{Lem:MConvEquiv} give
\begin{equation}\label{ConvInM}
	\boldV^m\ModConvM\boldV\text{ in }L^{M_y}(\Omega\times Y;\eR^{d\times N}),\ \boldA(\cdot,\boldV^m)\ModConv{M^*}\boldA(\cdot,\boldV) \text{ in }L^{M^*_y}(\Omega\times Y;\eR^{d\times N})\text{ as }m\to\infty.
\end{equation}
Let us consider a standard mollifier $\omega\in C^\infty(\eR^d\times\eR^d)$. Since $\boldV^m$ is supported in $K^m\subset\Omega$ for all $m$, we can find for every $m$ a sequence $\delta^n\to 0$ as $n\to\infty$ such that, defining $\boldV^{m,n}:=\boldV^m*\omega^n$, where $\omega^n(z)=(\delta^n)^{-2d}\omega\left(\frac{z}{\delta^n}\right)$, we have $\boldV^{m,n}\in C^\infty_c(\Omega;C^\infty_{per}(Y))^{d\times N}$. Obviously $\|\boldV^{m,n}\|_{L^\infty(\Omega\times Y)}\leq\|\boldV^m\|_{L^\infty(\Omega\times Y)}$.

  In the same manner as above we show that for every $m$
\begin{equation}\label{ConvInN}
	\boldV^{m,n}\ModConvM\boldV^m\text{ in }L^{M_y}(\Omega\times Y;\eR^{d\times N}),\ \boldA(\cdot,\boldV^{m,n})\ModConv{M^*}\boldA(\cdot,\boldV^m) \text{ in }L^{M^*_y}(\Omega\times Y;\eR^{d\times N})\text{ as }n\to\infty.
\end{equation}
Finally, using~\eqref{ConvInM},~\eqref{ConvInN}  and Lemma~\ref{Lem:ProdConv} we infer from Step 2 that
\begin{equation*}
	0\leq \lim_{m\to\infty}\lim_{n\to\infty}\int_\Omega\int_Y (\boldA^0-\boldA(y,\boldV^{m,n}))\cdot(\nabla \boldu+\boldU-\boldV^{m,n})=\int_\Omega\int_Y (\boldA^0-\boldA(y,\boldV))\cdot(\nabla \boldu+\boldU-\boldV).
\end{equation*}
\\
\textbf{Step 4}: For $k>0$ define
\begin{equation*}
S_k=\{(x,y)\in\Omega\times Y:|\nabla \boldu(x)+\boldU(x,y)|\leq k\}
\end{equation*}
and $\chi_k$ be the characteristic function of $S_k$. We replace $\boldV\in L^\infty(\Omega\times Y;\eR^{d\times N})$ in \eqref{LimMonIneq} by $(\nabla \boldu+\boldU)\chi_j+h\boldV\chi_i$ where $0<i<j$ and $h\in(0,1)$ to obtain
\begin{equation*}
	\begin{split}
	0\leq& \int_\Omega\int_Y \boldA^0\cdot(\nabla \boldu+\boldU-(\nabla \boldu-\boldU)\chi_j)\dy\dx\\
	&-\int_\Omega\int_Y\boldA(y,(\nabla \boldu-\boldU)\chi_j+h\boldV\chi_i))\cdot(\nabla \boldu+\boldU-(\nabla \boldu+\boldU)\chi_j)\\
	&+h\int_\Omega\int_Y(\boldA^0-\boldA(y,(\nabla \boldu+\boldU)\chi_j+h\boldV\chi_i))\cdot\boldV\chi_i\dy\dx.	\end{split}
\end{equation*}
The first term on the right-hand side vanishes  when passing to the limit with  $j\rightarrow\infty$ by the Lebesgue dominated convergence theorem and the fact that $|\Omega\times Y\setminus S_j|\rightarrow 0$ as $j\rightarrow \infty$. Since  $(\nabla \boldu+\boldU)\chi_j+h\boldV\chi_i=0$  in $S_j$, thus also the second term vanishes. After dividing the resulting inequality by $h$ we arrive at
\begin{equation}\label{IntPos}
	\int_{S_i}(\boldA^0-\boldA(y,\nabla \boldu+\boldU+h\boldV))\cdot\boldV\dy\dx\geq 0.
\end{equation}
By \ref{MTh} we obtain
\begin{equation}\label{UnifEstOnSI}
\begin{split}
	\int_{S_i} &M^*(y,\boldA(y,\nabla \boldu+\boldU+h\boldV))\dy\dx\leq \int_{S_i} m_1^*(|\boldA(y,\nabla \boldu+\boldU+h\boldV)|)\dy\dx\\
	&\leq |S_i|m_1^*(\|\boldA(\cdot,\nabla \boldu+\boldU+h\boldV)\|_{L^\infty(S_i)})\leq c.
	\end{split}
\end{equation}
We need to  estimate  $\|\boldA(\cdot,\nabla \boldu+\boldU+h\boldV)\|_{L^\infty(S_i)}$ uniformly with respect to $h$. For this purpose we proceed in a similar way as in  \eqref{AVMUnifBound} since
\begin{equation*}
\|\nabla \boldu+\boldU+h\boldV\|_{L^\infty(S_i)}\leq \|\nabla \boldu+\boldU\|_{L^\infty(S_i)}+\|\boldV\|_{L^\infty(\Omega\times Y)}\leq i+\|\boldV\|_{L^\infty(\Omega\times Y)}.
\end{equation*}
As $\boldA(y,\nabla \boldu+\boldU+h\boldV)\rightarrow\boldA(y,\nabla \boldu+\boldU)$ a.e. in $S_i$ and $\{\boldA(y,\nabla \boldu+\boldU+h\boldV)\}_{h\in(0,1)}$ is uniformly integrable on $S_i$ due to \eqref{UnifEstOnSI}  then again by the Vitali theorem we conclude
\begin{equation*}
	\boldA(y,\nabla \boldu+\boldU+h\boldV) \rightarrow \boldA(y,\nabla \boldu+\boldU)\text{ in }L^1(S_i)
\end{equation*}
as
 $h\rightarrow 0_+$. Thus passing to the limit  in \eqref{IntPos} we arrive at
\begin{equation*}
	\int_{S_i}(\boldA^0-\boldA(y,\nabla \boldu+\boldU))\cdot\boldV\dy\dx\geq 0.
\end{equation*}
Choosing
\begin{equation*}
	\boldV=-\frac{\boldA^0-\boldA(y,\nabla \boldu+\boldU)}{|\boldA^0-\boldA(y,\nabla \boldu+\boldU)|+1}
\end{equation*}
yields
\begin{equation}\label{PointEq}
	\boldA^0(x,y)=\boldA(y,\nabla \boldu(x)+\boldU(x,y))
\end{equation}
for a.a. $(x,y)\in S_i$. Since $i$ was arbitrary and $|\Omega\times Y\setminus S_i|\rightarrow 0$ as $i\rightarrow \infty$, the above holds a.e. in $\Omega\times Y$.
Moreover, due to the properties \eqref{BUReg} and \eqref{AzDual} we obtain that $\boldU(x,\cdot)$ is equal to the gradient of a weak solution of the cell problem \eqref{CellPr} corresponding to $\boldxi=\nabla \boldu(x)$. Finally, we get by \eqref{AZAv} and \eqref{AhDef} that
\begin{equation}\label{BAIdent}
	\bar\boldA(x)=\int_Y\boldA^0(x,y)\dy=\int_Y \boldA(y,\nabla \boldu(x)+\boldU(x,y))\dy=\hat\boldA(\nabla \boldu(x)).
\end{equation}
	\textbf{Step 5}:
	Since we know that \eqref{HPr} possesses a unique solution $\boldu$ and we can extract from any subsequence of $\{\boldu^j\}_{j=1}^\infty$ a subsequence that converges to $\boldu$ weakly$^*$ in $W^1_0L^{m_1}(\Omega;\eR^N)$ (thus also weakly in $W^{1,1}_0(\Omega;\eR^N)$), the whole sequence $\{\boldu^j\}_{j=1}^\infty$ converges to $\boldu$ weakly$^*$ in $W^1_0L^{m_1}(\Omega;\eR^N)$, weakly in $W^{1,1}_0(\Omega;\eR^N)$ respectively.


\begin{appendix}
\section{Musielak--Orlicz spaces}\label{Ape1}
Assume here that $\Sigma\subset\eR^n$ is a bounded domain and $n\in \mathbf{N}$ is arbitrary. A function $M:\Sigma\times\eR^n\rightarrow[0,\infty)$ is said to be an ${\mathcal N}-$function if it satisfies the following four requirements:
	\begin{enumerate}
		\item $M$ is a Carath\'eodory function such that $M(x,\boldxi)=0$ if and only if $\boldxi=\bzero$. In addition we assume that for almost all $x\in \Sigma$, we have  $M(x,\boldxi)=M(x,-\boldxi)$.
		\item For almost all $x\in \Sigma$ the mapping $\boldxi\mapsto M(x,\boldxi)$ is convex.
		\item For almost all $x\in \Sigma$ there holds $\lim_{\substack{|\boldxi|\rightarrow\infty}}\frac{M(x,\boldxi)}{|\boldxi|}=\infty$.
		\item For almost all $x\in \Sigma$ there holds $\lim_{|\boldxi|\rightarrow 0}\frac{M(x,\boldxi)}{|\boldxi|}=0$.
	\end{enumerate}		
The corresponding complementary ${\mathcal N}$--function $M^*$ to $M$ is defined for $\boldeta\in\eR^n$ and almost all  $x\in \Sigma$ by
\begin{equation*}
		M^*(x,\boldeta):=\sup_{\boldxi\in\eR^n}\{\boldxi\cdot\boldeta-M(x,\boldxi)\}
\end{equation*}
and directly from this  definition, one obtains the generalized Young inequality
\begin{equation}\label{YIneq}
		\boldxi\cdot\boldeta\leq M(x,\boldxi)+M^*(x,\boldeta),
\end{equation}
valid for all $\boldxi,\boldeta\in\eR^n$ and almost everywhere in $\Sigma$. In addition, for $\boldxi:=\nabla_{\boldeta}M^*(x,\boldeta)$, we obtain the equality sign in~\eqref{YIneq}, see \cite[Section 5]{SI69}. Finally, an ${\mathcal N}$-function $M$ is said to satisfy the $\Delta_2$--condition if there exists $c>0$ and a nonnegative function $h\in L^1(\Sigma)$ such that for a.a. $x\in\Sigma$ and all $\boldxi\in\eR^n$
	\begin{equation*}
		M(x,2\boldxi)\leq cM(x,\boldxi)+h(x).
	\end{equation*}

Having introduced the notion of an $\mathcal{N}$--function, we can define the generalized Musielak--Orlicz class	$\mathcal{L}^M(\Sigma)$ as a set of all measurable functions $\boldv:\Sigma\rightarrow\eR^n$ in the following way
\begin{equation*}
		\mathcal{L}^M(\Sigma):=\left\{\boldv \in L^1(\Sigma;\eR^n); \; \int_\Sigma M(x,\boldv(x))\dx<\infty\right\}.
\end{equation*}
In general the class $\mathcal{L}^M(\Sigma)$ does not form a linear vector space and therefore, we define the generalized Musielak--Orlicz space $L^M(\Sigma)$ as the smallest linear space containing $\mathcal{L}^M(\Sigma)$. More precisely, we define
\begin{equation*}
		L^M(\Sigma):=\left\{\boldv \in L^1(\Sigma;\eR^n); \; \textrm{ there exists $\lambda>0$ such that }\int_\Sigma M\left(x,\frac{\boldv(x)}{\lambda}\right)\dx<\infty\right\}.
\end{equation*}
It can be shown that $L^M(\Sigma)$ is a Banach space with respect to the Orlicz norm
	\begin{equation*}
		\|\boldv\|_{L^M}:=\sup\left\{\left|\int_{\Sigma}\boldv(x)\boldw(x)\dx\right|\!:\boldw\in L^{M^*}(\Sigma), \int_{\Sigma}M^*(x,\boldw(x))\dx\leq 1 \right\}
	\end{equation*}
or the equivalent Luxemburg norm
\begin{equation*}
		\|\boldv\|_{L^M}:=\inf\left\{\lambda>0:\int_\Sigma M\left(x,\frac{\boldv(x)}{\lambda}\right)\dx\leq 1\right\}.
\end{equation*}
Moreover, we have the following generalized H\"{o}lder inequality, see \cite[Theorem 4.1.]{SII69},
\begin{equation*}
		\left|\int_\Sigma \boldv\cdot\boldw \dx \right|\leq 2\|\boldv\|_{L^M}\|\boldw\|_{L^{M^*}}
\end{equation*}
valid for all $\boldv\in L^M(\Sigma)$ and all $\boldw\in L^{M^*}(\Sigma)$. It is not difficult to observe directly from the definition (or by Young inequality~\eqref{YIneq}), that
\begin{equation}\label{SSF}
		\|\boldv\|_{L^M}\leq c\left(\int_{\Sigma}M(x,\boldv(x))\dx+1\right),
\end{equation}
with some $c>0$, that can be set $c=1$ if we work with the Orlicz norm. Similarly, for the functional $\mathcal{F}:L^M(\Sigma)\rightarrow\eR$ defined as
\begin{equation*}
		\mathcal{F}(\boldv):=\int_\Sigma M(x,\boldv(x))\dx,
\end{equation*}
we can directly obtain from the definition and due to the convexity of $M$ that if $\|\boldv\|_{L^M}\leq 1$ and as the Luxemburg norm is considered then
\begin{equation}
\label{SSS}
\mathcal{F}(\boldv)\leq\|\boldv\|_{L^M}.
\end{equation}
Finally, we also recall the definition of the conjugate functional $\mathcal{F}^*:L^{M^*}(\Sigma)\rightarrow\eR$
$$
\mathcal{F}^*(\boldw):= \sup_{\boldv \in L^M(\Sigma)} \left(\int_{\Sigma} \boldv \cdot \boldw \dx - \mathcal{F}(\boldv) \right)
$$
and it is not difficult to observe by using the Young inequality that\footnote{Young inequality~\eqref{YIneq} implies $\mathcal{F}^*(\boldw)\leq \int_\Sigma M^*(x,\boldw(x))\dx$. On the other hand, we have $M^*(\cdot,\boldw)=\boldw\cdot\boldv-M(\cdot,\boldv)$ for $\boldv(x):=\nabla_{\boldxi}M^*(x,M^*(x,\boldw))$, which after integration leads to $\mathcal{F}^*(\boldw)\ge \int_\Sigma M^*(x,\boldw(x))\dx$ and~\eqref{SST} follows.}
\begin{equation}
\label{SST}
\mathcal{F}^*(\boldw) =\int_\Sigma M^*(x,\boldw(x))\dx.
\end{equation}

We complete this subsection by recalling the basic functional-analytic facts about the generalized Musielak--Orlicz spaces. For this purpose we define an additional space
\begin{equation*}
E^M(\Sigma):= \overline{\left\{L^{\infty}(\Sigma;\eR^n)\right\}}^{\|\cdot \|_{L^M(\Sigma)}}.
\end{equation*}
The following key lemma summarizes the fundamental properties of the involved function spaces (see e.g. \cite{S05} for details).
\begin{Lemma}[separability, reflexivity]\label{Thm:OrlSpProp}
	Let $M$ be an ${\mathcal N}$--function. Then
	\begin{enumerate}
		\item $E^M(\Sigma)=L^M(\Sigma)$ if and only if $M$ satisfies the $\Delta_2$--condition,
		\item $(E^M(\Sigma))^*=L^{M^*}(\Sigma)$, i.e., $L^{M^*}(\Sigma)$ is a dual space to $E^M(\Sigma)$,
		\item $E^M(\Sigma)$ is separable,
		\item $L^M(\Sigma)$ is separable if and only if $M$ satisfies the $\Delta_2$--condition,
		\item $L^M(\Sigma)$ is reflexive if and only if $M,M^*$ satisfy the $\Delta_2$--condition.
	\end{enumerate}
\end{Lemma}
We see from the above lemma that in some cases we need to face the problem with the density of bounded functions and also the lack of reflexivity and separability properties, that somehow excludes many analytical framework to be used. Thus, in addition to the strong/weak/weak$^*$ topology, we will also work with the modular topology. We say that a sequence $\{\boldv^k\}_{k=1}^\infty\subset L^M(\Sigma)$ converges modularly to $\boldv$ in $L^M(\Sigma)$ if there is $\lambda>0$ such that as $k\rightarrow\infty$
	\begin{equation*}
		\int_{\Sigma} M\left(x,\frac{\boldv^k(x)-\boldv(x)}{\lambda}\right)\dx\rightarrow 0.
	\end{equation*}
We use the notation $\boldv^k\ModConvM\boldv$ for the modular convergence in $L^M(\Sigma)$. The properties concerning the modular convergence are stated in the following lemmas.

\begin{Lemma}\cite[Proposition 2.2.]{GSG08}\label{Lem:ProdConv}
Let $M$ be an ${\mathcal N}$--function and $M^*$ be the conjugate ${\mathcal N}$--function to $M$. Suppose that sequences $\{\boldv^k\}_{k=1}^\infty$ and $\{\boldw^k\}_{k=1}^\infty$ are uniformly bounded in $L^M(\Sigma)$, $L^{M^*}(\Sigma)$ respectively. Moreover, let $\boldv^k\ModConvM\boldv$ and $\boldw^k\ModConv{M^*}\boldw$. Then $\boldv^k\cdot\boldw^k\rightarrow\boldv\cdot\boldw$ in $L^1(\Sigma)$ as $k\rightarrow\infty$.
\end{Lemma}
\begin{Lemma}\cite[Lemma 2.2.]{GSG08}\label{Lem:UnifIntegr}
Let $M$ be an ${\mathcal N}$--function and assume that there is $c>0$ such that $\int_\Sigma M(x,\boldv^k)\dx\leq c$ for all $k\in\eN$. Then $\{\boldv^k\}_{k=1}^\infty$ is uniformly integrable.
\end{Lemma}
\begin{Lemma}{\cite[Lemma 2.1.]{GSG08}}\label{Lem:MConvEquiv}
	Let $N\geq 1$, $M$ be an $\mathcal{N}$--function and $\{\boldv^k\}_{k=1}^\infty$ be a sequence of measurable $\eR^N-$valued functions on $\Sigma$. Then $\boldv^k\ModConvM\boldv$ in $L^M(\Sigma;\eR^N)$ if and only if $\boldv^k\rightarrow\boldv$ in measure and there exists some $\lambda>0$ such that $\{M(\cdot,\lambda \boldv^k)\}_{k=1}^\infty$ is uniformly integrable, i.e.,
	\begin{equation*}
		\lim_{R\rightarrow\infty}\left(\sup_{k\in\eN}\int_{\{x:|M(x,\lambda\boldv^k(x))|>R\}}M(x,\lambda\boldv^k(x))\dx\right)=0.
	\end{equation*}
\end{Lemma}
\begin{Lemma}\cite[Lemma 2.3]{BGKSG17}\label{Lem:ModConvGrTrunc}
	Let $M$ be an ${\mathcal N}$--function and $\Sigma$ be a bounded domain. Then for any  $v\in V_0^M(\Sigma)$ we have $\nabla T_k(v)\ModConvM\nabla v$ as $k\rightarrow\infty$, where
	\begin{equation*}
		T_k (v)=\begin{cases}
			v &\text{ if }|v|\leq k\\
			k\frac{v}{|v|} &\text{ if }|v|> k.
		\end{cases}
	\end{equation*}
\end{Lemma}

We continue with the lemma that provides the characterization of the space $E^M$.
\begin{Lemma}[Lemma~4.4, \cite{BGKSG17}]\label{Lem:EMChar}
Let $\Sigma\subset\Rd$ be bounded, $M$ be an ${\mathcal N}$--function such that for all $R>0$
\begin{equation}\label{MLocBound}
	\int_{\Sigma}\sup_{|\boldxi|\leq R}M(x,\boldxi)\dx<\infty.
\end{equation}
Then
\begin{equation*}
	E^M(\Sigma)=\{\boldv\in L^M(\Sigma):\; \textrm{ for all } t\geq 0 \textrm{ we have } t\boldv\in\mathcal{L}^M(\Sigma)\}.
\end{equation*}
\end{Lemma}
We complete this section with the lemma that concerns the uniform boundedness of a composition of a mapping possessing the Orlicz growth and a sequence of bounded functions.
\begin{Lemma}\label{Lem:UBSeq}
	Let $d\geq 2$, $N\geq 1$ and for a Carath\'eodory mapping $\boldE:\eR^d\times\eR^{d\times N}\to\eR^{d\times N}$ let there exist an $\mathcal{N}$--function $M$ such that for a.a. $x\in\Rd$ and all $\boldxi\in\eR^{d\times N}$
	\begin{equation}\label{ECond}
		\boldE(x,\boldxi)\cdot\boldxi\geq c(M(x,\boldxi)+M^*(x,\boldxi))
	\end{equation}
	holds for some constant $c$. Let $\{\boldV^n\}_{n=1}^\infty$ be bounded in $L^\infty(\eR^d;\eR^{d\times N})$. Then $\{\boldE(\cdot,\boldV^n)\}_{n=1}^\infty$ is bounded in $L^\infty(\eR^d;\eR^{d\times N})$.
	Moreover, if $\Omega\subset\eR^d$ is bounded and measurable, $Y=(0,1)^d$, then
	\begin{enumerate}
		\item for any $\boldV\in L^\infty(\Omega\times Y;\eR^{d\times N})$ we have $\boldE(\cdot,\boldV)\in L^\infty(\Omega\times Y;\eR^{d\times N})$,
		\item for any $\boldV\in E^{M_y}(\Omega\times Y;\eR^{d\times N})$ we have $\boldE(\cdot,\boldV)\in E^{M_y^*}(\Omega\times Y;\eR^{d\times N})$ provided \ref{MTh} holds.
		\end{enumerate}
		\begin{proof}
			Assuming on the contrary that $\{\boldE(\cdot,\boldV^n)\}_{n=1}^\infty$ is unbounded, we have for arbitrary $K>0$ the existence of $n_K>0$ and $S_K\subset\eR^d$ with $|S_K|>0$ such that $|\boldE(\cdot,\boldV^{n_K})|>K$ on $S_K$. As $M$ is an $\mathcal{N}$--function, for a chosen $C>0$ there is $R>0$ such that $\frac{M(x,\boldxi)}{|\boldxi|}>C$ for any $|\boldxi|\geq R$. Thus for the choice $C=\sup_{n\in\eN}\|\boldV^n\|_{L^\infty(\eR^d)}$ we find $n_{R}$ and $S_{R}\subset\Omega\times Y$  with $|S_{R}|>0$ such that for $x\in S_{R}$ we obtain using~\eqref{ECond} that
\begin{equation*}
	C< \frac{M^{*}(x,\boldE(x,\boldV^{n_{R}}))}{|\boldE(x,\boldV^{n_{R}})|}\leq |\boldV^{n_{R}}|\leq C,
\end{equation*}
which contradicts the unboundedness of $\{\boldE(\cdot,\boldV^n)\}_{n=1}^\infty$.

For the proof of the second part of the lemma we refer to \cite[Lemma 3.4.]{BGKSG17}
	\end{proof}
\end{Lemma}

\section{Auxiliary tools} \label{Ape2}
\begin{Lemma}\label{Lem:Duality}
	Let $X$ be a Banach space, $V$ be a subspace of $X$, $g$ be a closed, convex functional on $X$ that is continuous at some $x\in V$.
	Then
	\begin{equation}\label{IdDual}
     \inf_{x\in V} \{g(x)-\langle \eta,x\rangle\}+\inf_{\xi\in V^\bot} g^*(\eta+\xi)=0
	\end{equation}
	for all  $\eta\in X^*$.
	\begin{proof}
	One deduces by definition of a convex conjugate that	
	\begin{equation}\label{GConj}
\forall\xi\in X^*: (g-\eta)^*(\xi)=\sup_{x\in X}\{\langle \eta+\xi,x\rangle-g(x)\}=g^*(\eta+\xi).
	\end{equation}
	According to \cite[Theorem 14.2]{ZKO94}
	\begin{equation*}
		\inf_{x\in V} A(x)+\inf_{x^*\in V^\bot}A^*(x^*)=0
	\end{equation*}
	for a closed, convex functional $A$ that is continuous at some $x\in V$. We set $A(x):=(g-\eta)(x)$ and the expression for $A^*$ determined by~\eqref{GConj} in the latter equality to conclude~\eqref{IdDual}.
	\end{proof}
\end{Lemma}
Finally, we also recall the weak$^*$ lower semicontinuity property of convex functionals. Since in our case,  the ${\mathcal N}$--function $M$ may not  satisfy the $\Delta_2$--condition in general, the spaces  do not have to be  reflexive. However, due to Lemma~\ref{Thm:OrlSpProp}, we see that any $L^M$ always has a separable predual space and consequently any bounded sequence possesses a weakly$^*$ convergent subsequence. This motivates us to introduce the last convergence theorem, that can be obtained by standard  weak lower semicontinuity properties of convex functionals, see e.g. \cite[Theorem 4.5]{G03}, namely:
\begin{Lemma}\label{Lem:SemCon}
	Let $\Omega\subset\Rd$ be open, $Y=(0,1)^d$, $n\in\eN$ and $\Phi:Y\times \eR^n \rightarrow \eR$ satisfy:
\begin{enumerate}[label=(\alph*)]
	\item $\Phi$ is Carath\'eodory,
	\item $\Phi(y,\cdot)$ is convex for almost all $y\in Y$,
	\item $\Phi\geq 0$.
\end{enumerate}
	Then we have the following semicontinuity property: $\boldv^k\WCon \boldv$ in $L^1(\Omega\times Y;\eR^n)$ as $k\rightarrow\infty$ implies
	\begin{equation*}
		\liminf_{k\rightarrow\infty}\int_{\Omega}\int_Y \Phi(y,\boldv^k(x,y))\dy\dx\geq\int_\Omega\int_Y \Phi(y,\boldv(x,y))\dy\dx.
		\end{equation*}
\end{Lemma}


\section{Existence of solutions to elliptic problems}\label{Ape3}
To the best of authors' knowledge only the result from \cite{GMW12, GSZG2017} concern  the existence of weak solutions of elliptic problems in which the growth condition is given by an anisotropic inhomogeneous ${\mathcal N}-$function. For the sake of completeness, we show here that for elliptic systems it is possible to obtain existence of a weak solution provided one of the conditions \ref{MDeltaTwo}-\ref{MStDeltaTwo} holds.

Before we prove existence results corresponding to the conditions \ref{MDeltaTwo} and \ref{MStDeltaTwo}, we state a variant of the Brouwer fixed point theorem.
\begin{Lemma}\cite[p.493]{E98}\label{Lem:FixP}
	Let $\bolds:\eR^m\rightarrow\eR^m$ be a continuous mapping and
	\begin{equation}\label{NonNegCond}
		\bolds(x)\cdot x\geq0 \text{ if } |x|=r
	\end{equation}
	for some $r>0$. Then there is a point $x$ with $|x|\leq r$ such that $\bolds(x)=0$.
\end{Lemma}

\begin{proof}[Proof of Lemma~\ref{Lem:ExUnMStDTwo}]
		Let us note that as $M^*$ satisfies $\Delta_2-$condition we have $L^{M^*}(\Omega)=E^{M^*}(\Omega)$ by Lemma~\ref{Thm:OrlSpProp}. We proceed by the Galerkin method. First, we observe that $W^1_0E^M(\Omega;\eR^{N})$ is separable as it is a closed subspace of the separable space $E^M(\Omega;\eR^{N})^{d\times N}$. Thus there exists $\{\boldw^i\}^\infty_{i=1}$ a linearly independent subset of $W^1_0E^M(\Omega;\eR^{N})$, such that $\overline{\bigcup_{k=1}^\infty Lin\{\boldw^i\}_{i=1}^k}^{\|\cdot\|_{W^{1,M}}}=W^1_0E^M(\Omega;\eR^{N})$. We construct Galerkin approximations to \eqref{EllPrWeakFormMStDTw}. Let us define $\boldu^k=\sum_{i=1}^k\alpha^k_i\boldw^i$ for $k\in\eN$, where $\alpha_i^k\in\eR$ are determined by
		\begin{equation}\label{GalAppMStDTwo}
			\int_\Omega \boldA\left(x,\nabla \boldu^k\right)\cdot\nabla \boldw^i\dx=\int_\Omega \boldF\cdot\nabla \boldw^i\dx
		\end{equation}
		for all $i=1,\ldots,k$.\\
	\textbf{Step 1}:
		We show the existence of $(\alpha_1^k,\ldots,\alpha_k^k)\in\eR^k$ satisfying \eqref{GalAppMStDTwo} with the help of Lemma \ref{Lem:FixP}. We define a mapping $\bolds:\eR^k\rightarrow \eR^k$ as
	\begin{equation*}
	s_j(\balpha)=\int_\Omega \boldA\left(x,\sum_{i=1}^k\alpha_i \nabla \boldw^i\right)\cdot\nabla \boldw^j-\boldF\cdot\nabla \boldw^j\dx.
	\end{equation*}
	Let us denote $\boldW(\balpha):=\sum_{i=1}^k\alpha_i \nabla \boldw^i$. First, we show that $\bolds$ is continuous. Let $\balpha^n\rightarrow \balpha$ in $\eR^k$. Then
	\begin{equation*}
		|s_j(\balpha^n)-s_j(\balpha)|=\left|\int_\Omega \left(\boldA\left(x,\boldW(\balpha^n)\right)-\boldA\left(x,\boldW(\balpha)\right)\right)\cdot\nabla \boldw^j\dx\right|
	\end{equation*}
	for all $j=1,\ldots,k$.	We denote
	\begin{equation*}
		h^n_j:=\left(\boldA\left(x,\boldW(\balpha^n)\right)-\boldA\left(x,\boldW(\balpha)\right)\right)\cdot\nabla \boldw^j.
	\end{equation*}
	Obviously, we have for almost all $x\in\Omega$ that $h^n_j\rightarrow 0$ as $n\rightarrow\infty$.
 From \ref{ATh} and the Young inequality it follows that
	\begin{equation*}
	\begin{split}
		c&\int_\Omega M\left(x,\boldW(\balpha^n)\right)+M^*\left(x,\boldA\left(x,\boldW(\balpha^n)\right)\right)\dx\leq\int_\Omega \boldA\left(x,\boldW(\balpha^n)\right)\cdot \boldW(\balpha^n)\dx\\&\leq \int_\Omega M\left(x,\frac{2}{c}\boldW(\balpha^n)\right)+\frac{c}{2}M^*\left(x,\boldA\left(x,\boldW(\balpha^n)\right)\right)\dx.
		\end{split}
	\end{equation*}
	Hence we obtain by convexity of $M$ with respect to the second variable
	\begin{equation*}
		\begin{split}
		\frac{c}{2}&\int_\Omega M^*\left(x,\boldA\left(x,\boldW(\balpha^n)\right)\right)\dx\leq \int_\Omega M\left(x,\frac{2}{c}\boldW(\balpha^n)\right)\dx\\&\leq \sum_{i=1}^k\frac{\alpha_i^{n}}{|\balpha^{n}|}\int_\Omega M\left(x,\frac{2}{c}|\balpha^{n}|\nabla \boldw^i\right)\dx\leq k\max_{i=1,\ldots k} \int_\Omega M\left(x,\frac{2}{c}|\balpha^{n}|\nabla \boldw^i\right)\dx,
			\end{split}
		\end{equation*}
		which is finite as $\{\balpha^{n}\}_{n=1}^\infty$ is bounded. From the latter estimate one deduces the uniform integrability of $\boldA\left(\cdot, \boldW(\balpha^n)\right)$. As $\boldA(\cdot,\boldW(\balpha))\in L^1(\Omega;\eR^{d\times N})$ and $\nabla \boldw^j\in L^\infty\left(\Omega;\eR^{d\times N}\right)$, we have that $h^n_j$ is uniformly integrable. Hence we conclude the continuity of $\bolds$ since by the Vitali theorem it follows that
		\begin{equation*}
			|\bolds(\balpha^n)-\bolds(\balpha)|\leq k\max_{j=1,\ldots,k}\int_\Omega |h^n_j|\dx\rightarrow 0\text{ as }n\rightarrow \infty.
		\end{equation*}
	Now, we show that $\bolds$ satisfies \eqref{NonNegCond}.  Employing \ref{ATh}, the Young inequality and \eqref{SSF} we deduce
		\begin{equation}\label{FixPoinPrep}
			\begin{split}
			\bolds(\balpha)\cdot\balpha&=\int_\Omega \boldA\left(x,\boldW(\balpha)\right)\cdot \boldW(\balpha)-\boldF\cdot \boldW(\balpha)\dx\geq \frac{c}{2}\int_\Omega M\left(x,\boldW(\balpha)\right)\dx-M^*\left(x,\frac{2}{c}\boldF\right)\dx\\
			&\geq c\|\boldW(\balpha)\|_{L^M(\Omega)}-1-\int_\Omega M^*\left(x,\frac{2}{c}\boldF\right)\dx.
			\end{split}
		\end{equation}
		Let us show that
		\begin{equation}\label{FinAppUnboun}
			\left\|\boldW(\balpha)\right\|_{L^M(\Omega)}\rightarrow\infty\text{ as }|\balpha|\rightarrow\infty.
		\end{equation}
		We observe that $\balpha\mapsto\|\boldW(\balpha)\|_{L^M(\Omega)}$ is a continuous function, in particular it is continuous on the unit sphere $S_1$ in $\eR^k$, which is compact. Thus the minimum of $\|\boldW(\balpha)\|_{L^{M^*}(\Omega)}$ on $S_1$  is attained at some $\bbeta\in S_1$. We show that $\|\boldW(\bbeta)\|_{L^{M^*}(\Omega)}>0$. Assume contrary that $\|\boldW(\bbeta)\|_{L^M(\Omega)}=0$. Then we have $\left\|\sum_{i=1}^k\beta_i\nabla \boldw^i\right\|_{L^1(\Omega)}=0$ and by the Poincar\'e inequality $\|\sum_{i=1}^k\beta_i \boldw^i\|_{L^1(\Omega)}=0$. Hence we obtain $\sum_{i=1}^k\beta_i \boldw^i=0$ a.e. in $\Omega$, which implies $\beta_i=0$ for each $i=1,\ldots,k$ since $\{\boldw^i\}_{i=1}^k$ are linearly independent, which is a contradiction. Thus $\|\boldW(\bbeta)\|_{L^{M^*}(\Omega)}>0$, we have $\left\|\boldW(\balpha)\right\|_{L^M(\Omega)}=|\balpha|\left\|\boldW\left(\frac{\balpha}{|\balpha|}\right)\right\|_{L^M(\Omega)}\geq |\balpha|\|\boldW(\bbeta)\|_{L^{M^*}(\Omega)}
 $ and \eqref{FinAppUnboun} follows. For $R$ large enough we obtain that $\bolds(\balpha)\cdot\balpha\geq 0$ for $|\alpha|=R$ from~\eqref{FixPoinPrep}. Consequently, by Lemma \ref{Lem:FixP} there is $\balpha\in\eR^k$ satisfying \eqref{GalAppMStDTwo}.\\
	\textbf{Step 2}: We show uniform estimates of $\{\uk\}_{k=1}^\infty$ and $\{\boldA\left(x,\nabla u^k\right)\}_{k=1}^\infty$. Multiplying \eqref{GalAppMStDTwo} by $\alpha_i^k$ and summing over $i=1,\ldots,k$ yields
			\begin{equation}\label{MStDTwoTestAp}
			\int_\Omega \boldA\left(x,\nabla \uk\right)\cdot\nabla \uk\dx=\int_\Omega \boldF\cdot\nabla \uk\dx.
		\end{equation}	
		Hence we obtain using \ref{ATh} and the Young inequality
		\begin{equation*}
			\frac{c}{2}\int_\Omega M\left(x,\nabla \uk\right)\dx+c\int_\Omega M^*\left(x,\boldA\left(x,\nabla \uk\right)\right)\dx\leq \int_\Omega M^*\left(x,\frac{2}{c}\boldF\right)\dx.
		\end{equation*}
		Since the right hand side of the latter inequality is finite as $\boldF\in L^\infty(\Omega;\eR^{d\times N})$, we infer the existence of $\boldu\in W^1_0L^M(\Omega;\eR^{N})$ and $\bar\boldA\in E^{M^*}(\Omega;\eR^{d\times N})$ such that
		\begin{equation}\label{MStDTwAppConv}
			\begin{alignedat}{2}
				\nabla \uk&\WSCon\nabla \boldu&&\text{ in }L^M(\Omega;\eR^{d\times N}),\\
				\boldA(\cdot,\nabla \uk)&\WSCon \bar\boldA &&\text{ in }E^{M^*}(\Omega;\eR^{d\times N})
			\end{alignedat}
		\end{equation}
		as $k\rightarrow\infty$. Notice that \eqref{MStDTwAppConv} follows from the fact that $M^*$ satisfies $\Delta_2-$condition.\\
		\textbf{Step 3}: We identify the limit function $ \bar\boldA$. Employing the convergence \eqref{MStDTwAppConv}$_2$ in \eqref{GalAppMStDTwo} we have
		\begin{equation}\label{MK13}
			\int_\Omega \bar\boldA\cdot\nabla \boldw^i\dx=\int_\Omega \boldF\cdot\nabla \boldw^i\dx
		\end{equation}
		for each $i=1,\ldots,k$. Multiplying by $\alpha_i^k$ and summing over $i=1,\ldots,k$ we get
		\begin{equation*}
			\int_\Omega \bar\boldA\cdot\nabla \uk\dx=\int_\Omega \boldF\cdot\nabla \uk\dx.
		\end{equation*}
		Since $\bar\boldA\in E^{M^*}(\Omega;\eR^{d\times N})$, we obtain using the convergence \eqref{MStDTwAppConv}$_1$ 	
		\begin{equation}\label{MStDTwoIdentForLim}
			\int_\Omega \bar\boldA\cdot\nabla \boldu\dx=\int_\Omega \boldF\cdot\nabla \boldu\dx.
		\end{equation}
		Moreover, the application of \eqref{MStDTwAppConv}$_1$ in \eqref{MStDTwoTestAp} yields
		\begin{equation}\label{MStDTwoLimIdent}
			\lim_{k\rightarrow\infty}\int_\Omega \boldA\left(x,\nabla \uk\right)\cdot\nabla \uk\dx=\int_\Omega \boldF\cdot\nabla \boldu\dx.
		\end{equation}
		Let us choose an arbitrary $\boldW\in L^\infty(\Omega;\eR^{d\times N})$. The monotonicity of $\boldA$ combined with \eqref{MStDTwoTestAp} yields
		\begin{equation*}
			0\leq \int_\Omega \left(\boldA\left(x,\nabla \uk\right)-\boldA(x,\boldW)\right)\cdot(\nabla \uk-\boldW)\dx= \int_\Omega \boldF\cdot\nabla \boldu^k -\boldA(x,\nabla\boldu^k)\cdot\boldW-\boldA(x,\boldW)\cdot(\nabla \uk-\boldW)\dx.
		\end{equation*}
	We employ \eqref{MStDTwAppConv} to perform the limit passage $k\to\infty$ in the latter inequality and use \eqref{MStDTwoIdentForLim} to obtain
	\begin{equation}\label{SkoroMintyNer}
			0\leq \int_\Omega (\bar\boldA -\boldA(x,\boldW))\cdot(\nabla \boldu-\boldW)\dx.
		\end{equation}
		Then we denote for a positive $l$
		\begin{equation*}
			\Omega_l=\{x\in\Omega:|\nabla \boldu|\leq l\}
		\end{equation*}
		and $\chi_l$ be the characteristic function of $\Omega_l$. We choose arbitrary $0<l<m$, $h>0$ and $\boldZ\in L^\infty(\Omega;\eR^{d\times N})$ and set $\boldW=\nabla\boldu\chi_m+h\boldZ\chi_l$ in \eqref{SkoroMintyNer} to obtain
		\begin{equation}\label{SkoroSkoroMintyNer}
			0\leq \int_{\Omega\setminus\Omega_m}\bar\boldA\cdot\nabla\boldu\dx-h\int_{\Omega_l}(\bar\boldA -\boldA(x,\nabla\boldu+h\boldZ))\cdot\boldZ\dx.
		\end{equation}
		Then as $|\Omega\setminus\Omega_m|\to 0$ and $\bar\boldA\cdot\nabla\boldu\chi_{\Omega\setminus\Omega_m}\to 0$ as $m\to\infty$ we infer performing the limit passage $m\to\infty$ in \eqref{SkoroSkoroMintyNer} by the Lebesgue dominated convergence theorem
	\begin{equation}\label{SkoroSkoroSkoroMintyNer}
			0\leq -h\int_{\Omega_l}(\bar\boldA -\boldA(x,\nabla\boldu+h\boldZ))\cdot\boldZ\dx.
		\end{equation}	
		Next, by \eqref{IntMSphere2} we get
		\begin{equation*}
	\int_{\Omega_l} M^*(x,\boldA(\cdot,\nabla\boldu+h\boldZ))\dx\leq \int_\Omega\sup_{|\boldxi|=\|\boldA(\cdot, \nabla\boldu+h\boldZ)\|_{L^\infty(\Omega_l)}}M^*(x,\boldxi)\dx<\infty
\end{equation*}
uniformly in $h\in(0,1)$ as $\sup_{h\in(0,1)}\|\boldA(\cdot, \nabla\boldu+h\boldZ)\|_{L^\infty(\Omega_l)}<\infty$, which follows from Lemma~\ref{Lem:UBSeq}. Hence $\{\boldA(\cdot,\nabla\boldu+h\boldZ)\}_{h\in(0,1)}$ is uniformly integrable. Furthermore $\boldA(\cdot,\nabla\boldu+h\boldZ)\to \boldA(\cdot,\nabla\boldu)$ a.e. in $\Omega$ as $h\to 0$. Thus by the Vitali convergence theorem we get performing the limit passage $h\to 0$ in \eqref{SkoroSkoroSkoroMintyNer} divided by $-h$
	\begin{equation}\label{MintyNer}
			0\geq\int_{\Omega_l}(\bar\boldA -\boldA(x,\nabla\boldu))\cdot\boldZ\dx
		\end{equation}	
		for any $\boldZ\in L^\infty(\Omega;\eR^{d\times N})$. Setting
		\begin{equation*}
			\boldZ=\frac{\bar\boldA -\boldA(x,\nabla\boldu)}{1+|\bar\boldA -\boldA(x,\nabla\boldu)|}
		\end{equation*}
		in \eqref{MintyNer} we deduce $\bar\boldA(x)=\boldA(x,\nabla\boldu(x))$ for a.a. $x\in\Omega_l$. As $|\Omega\setminus\Omega_l|\to 0$ as $l\to\infty$ we infer $\bar\boldA(x)=\boldA(x,\nabla\boldu(x))$ for a.a. $x\in\Omega$.\\
		\textbf{Step 4}: We show the uniqueness of a weak solution. Supposing that $\boldu_1, \boldu_2$ are weak solutions satisfying \eqref{EllPrWeakFormMStDTw}, we subtract the weak formulations corresponding to $\boldu_1$ and $\boldu_2$ to obtain
		\begin{equation}\label{MStDTwWFDiff}
			\int_{\Omega}\left(\boldA(x,\boldu_1)-\boldA(x,\boldu_2)\right)\cdot\nabla\bphi\dx=0\ \forall \bphi\in W^{1}_0L^M(\Omega;\eR^{N}).
		\end{equation}
		As we have $\boldu_1-\boldu_2\in W^1_0 L^M(\Omega;\eR^N)$, wet set $\bphi:=\boldu_1-\boldu_2$ in \eqref{MStDTwWFDiff} to get
		\begin{equation*}
			\int_{\Omega}\left(\boldA(x,\boldu_1)-\boldA(x,\boldu_2)\right)\cdot(\nabla \boldu_1-\nabla \boldu_2)\dx=0.
		\end{equation*}
		Then \ref{AF} implies $\nabla (\boldu_1-\boldu_2)=0$ a.e. in $\Omega$. Regarding the zero trace of $\boldu_1-\boldu_2$ on $\partial\Omega$ we conclude $\boldu_1=\boldu_2$ a.e. in $\Omega$.
	\end{proof}
\end{appendix}


\def\cprime{$'$}
\providecommand{\bysame}{\leavevmode\hbox to3em{\hrulefill}\thinspace}
\providecommand{\MR}{\relax\ifhmode\unskip\space\fi MR }
\providecommand{\MRhref}[2]{%
  \href{http://www.ams.org/mathscinet-getitem?mr=#1}{#2}
}
\providecommand{\href}[2]{#2}


\end{document}